\documentclass[12pt]{amsart}  %[12pt] while developing, omitted for final
\usepackage[latin1]{inputenc}
\usepackage{amsmath} 
\usepackage{amsfonts}
\usepackage{amssymb}
\usepackage{stmaryrd}
\usepackage{latexsym} 
\usepackage{graphicx}
\usepackage{subfigure}
\usepackage{color}
\usepackage{hyperref}
\usepackage{verbatim}
\usepackage[all]{xy}
\usepackage{graphics}
\usepackage{pdfsync}
\usepackage{xcolor}
\usepackage{tikz}
\usepackage{bm}
\usetikzlibrary{arrows.meta}
\usetikzlibrary{shapes.geometric}
\usetikzlibrary{calc,math}
\usetikzlibrary{external}
\usetikzlibrary{decorations.markings}
\theoremstyle{definition}
\newtheorem{theorem}{Theorem}[section]
\newtheorem{proposition}[theorem]{Proposition}

\newtheorem{definition}[theorem]{Definition}
\newtheorem{example}[theorem]{Example}

\newtheorem{definition/theorem}[theorem]{Definition/Theorem}

\theoremstyle{remark}
\newtheorem{remark}[theorem]{Remark}

\numberwithin{equation}{section}
\newlength\cellsize 
\savebox2{%
\begin{picture}(15,15)
\put(0,0){\line(1,0){15}}
\put(0,0){\line(0,1){15}}
\put(15,0){\line(0,1){15}}
\put(0,15){\line(1,0){15}}
\end{picture}}
\newcommand\cellify[1]{\def\thearg{#1}\def\nothing{}%
\ifx\thearg\nothing
\vrule width0pt height\cellsize depth0pt\else
\hbox to 0pt{\usebox2\hss}\fi%
\vbox to 15\unitlength{
\vss
\hbox to 15\unitlength{\hss$#1$\hss}
\vss}}
\newcommand\tableau[1]{\vtop{\let\\=\cr
\setlength\baselineskip{-16000pt}
\setlength\lineskiplimit{16000pt}
\setlength\lineskip{0pt}
\halign{&\cellify{##}\cr#1\crcr}}}
\savebox3{%
\begin{picture}(15,15)
\put(0,0){\line(1,0){15}}
\put(0,0){\line(0,1){15}}
\put(15,0){\line(0,1){15}}
\put(0,15){\line(1,0){15}}
\end{picture}}
\newcommand\expath[1]{%
\hbox to 0pt{\usebox3\hss}%
\vbox to 15\unitlength{
\vss
\hbox to 15\unitlength{\hss$#1$\hss}
\vss}}
\newcommand\bas[1]{\omit \vbox to \cellsize{ \vss \hbox to \cellsize{\hss$#1$\hss} \vss}}

\DeclareMathOperator{\sign}{sign}
\DeclareMathOperator{\sort}{sort}
\DeclareMathOperator\type{type}

\DeclareMathOperator{\len}{\ell}

\DeclareMathOperator{\ctrip}{\mathcal{C}}
\DeclareMathOperator{\strip}{\mathcal{S}}
\DeclareMathOperator{\ST}{ST}
\DeclareMathOperator{\trace}{trace}
\DeclareMathOperator{\rev}{rev}

\NewDocumentCommand{\vertexs}{O{black} O{1cm} O{above} m O{(0,0)} m O{}}{
\path let \p1 = #6, \p2=#5 in node[style={draw,fill,color=#1,circle,minimum size=6pt,inner sep=0},label={[text=#1, label distance=#2]#3:}] (v#7#4) at (\x1+\x2,\y1+\y2){};}
\NewDocumentCommand{\vertex}{O{black} O{-3pt} O{above} m O{(0,0)} m O{}}{
\path let \p1 = #6, \p2=#5 in node[style={draw,fill,color=#1,circle,minimum size=6pt,inner sep=0},label={[text=#1, label distance=#2]#3:#4}] (v#7#4) at (\x1+\x2,\y1+\y2){};}
\NewDocumentCommand{\vertexl}{O{black} O{0pt} O{0pt} m O{(0,0)} m O{}}{
\path let \p1 = #6, \p2=#5 in node[style={draw,fill,color=#1,circle,minimum size=6pt,inner sep=0},label={[text=#1, shift={(#2, #3)}]#4}] (v#7#4) at (\x1+\x2,\y1+\y2){};}

\usepackage[backend=bibtex]{biblatex}
\begin{document}

\title[The Tutte symmetric matrix of a graph]{The Tutte symmetric matrix of a graph}

\author{Foster Tom}
\address{Department of Mathematics, Dartmouth College, Hanover NH 03755}
\email{foster.tom@dartmouth.edu}

\author{Aarush Vailaya}
\address{Department of Mathematics, Massachusetts Institute of Technology, Cambridge MA 02139}
\email{aarushv@mit.edu}

%05: Graph theory
%05C15:Coloring of graphs and hypergraphs
%05C50:Graphs and matrices
%05E05:Symmetric functions and generalizations
\subjclass[2020]{Primary 05C50; Secondary 05E05, 05C15}
\keywords{chromatic polynomial, chromatic quasisymmetric function, Tutte polynomial, Tutte symmetric function}

\begin{abstract}
We provide a matrix-based formula for the Tutte symmetric function of a graph. In particular, for any graph $G$ with a designated head and tail vertex, we describe an infinite matrix $M_G$ from which the Tutte symmetric function can be easily recovered. We prove gluing graphs together corresponds to matrix multiplication, gluing the head and tail of a single graph corresponds to taking the trace, and reversing a graph corresponds to the transpose (up to a change of basis).
\end{abstract}

\maketitle
%\tableofcontents
%\vspace{65pt}

\section{Introduction}
The Tutte symmetric function is a generalization of the well-studied Tutte polynomial and the chromatic symmetric function. The Tutte polynomial $T_G(x, y)$ is the most general graph function that satisfies the deletion-contraction relation. It generalizes the chromatic and flow polynomials, and has many connections to areas such as knot theory and quantum field theory \cite{tuttehandbook}. The chromatic symmetric function $X_G(\bm x)$ also generalizes the chromatic polynomial, encodes additional information about acyclic orientations, and has connections to Hessenberg varieties \cite{dothessenberg,chromquasihessenberg}, 
LLT polynomials \cite{lltchrom}, and Hecke algebras \cite{chromhecke}. Moreover, unlike the Tutte polynomial which is the same for trees with a given number of vertices, the chromatic symmetric function is conjectured to distinguish trees \cite{chromsym}. Another famous conjecture regarding the chromatic symmetric function was the Stanley--Stembridge conjecture, which claimed the chromatic symmetric function of unit interval graphs had non-negative coefficients in the $e$-basis, and was recently resolved by Hikita in \cite{stanstemproof}.

A lesser-studied generalization of both functions is the Tutte symmetric function $XB_G(\bm x; t)$ \cite{graphrelated}. Previous papers have studied this function to find pairs of graphs with equal Tutte and chromatic symmetric functions and to characterize the Tutte symmetric function as a general graph invariant satisfying properties similar to the familiar deletion-contraction relation \cite{tuttevtxweighted, tuttemodular}. Still, little is known about the structure or behavior of the Tutte symmetric function.

All graph functions mentioned above are multiplicative over connected components. Similarly, when joining two graphs at a single vertex, the Tutte polynomial is also multiplicative, but no analogous formula was known for the chromatic symmetric function. In \cite{gluesinglevertex}, the authors extend Hikita's recent proof of the Stanley--Stembridge conjecture to derive $e$-positive formulas for a varied family of graphs including certain proper circular-arc graphs.

This paper builds on \cite{gluesinglevertex}, which established a matrix-based formula for the chromatic symmetric function of graphs glued at a single vertex. We generalize this framework to the Tutte symmetric function by constructing an infinite matrix $M_G$ where gluing graphs at single vertices corresponds to matrix multiplication, identifying two vertices on the same graph corresponds to taking the trace of the matrix, and reversing a graph corresponds to taking the transpose (up to a change of basis).

This framework reveals that several natural properties of graphs correspond to fundamental properties of matrices. For instance, when three graphs $A$, $B$, $C$ are attached in a circle, the resulting graph $(A+B+C)^\circ$ is isomorphic to $(B+C+A)^\circ$ by the cyclic symmetry, but need not be $(A+C+B)^\circ$. This mirrors the fact that the trace is invariant under cyclic but not arbitrary reordering. Similarly, reversing a graph corresponds to finding a trace-preserving map $f$ satisfying $f(AB)=f(B)f(A)$, and the transpose (up to a change of basis) is the typical example.

More specifically, Section~\ref{section:background} provides a brief background of the Tutte symmetric function. In Section~\ref{section:tsm} we introduce the Tutte symmetric matrix and establish its basic properties, such as the multiplication property. Section~\ref{section:traceresult} shows that identifying two vertices of a graph is equivalent to taking the trace of the Tutte symmetric matrix, and Section~\ref{section:reverse} shows that reversing a graph is equivalent to transposing the matrix up to a change of basis.

\section{Tutte Symmetric Function}\label{section:background}
Graphs in this paper have vertex set $[n]=\{1,\ldots,n\}$ for some $n$. For graphs $G=([n],E)$ and $H=([n'],E')$ we define the \emph{sum}
\begin{equation*}
G+H=([n+n'-1],E\cup \{\{i+n-1,j+n-1\}:\{i,j\}\in E'\}).
\end{equation*} 
Informally, we glue vertex $n$ of $G$ to vertex $1$ of $H$.

\begin{figure}
$$\begin{tikzpicture}
\node at (-0.75,0) {$G=$};
% Triangle pointing right (|>)
\draw (0,-0.5)--(0,0.5)--({sqrt(3)/2},0)--(0,-0.5);
\filldraw (0,-0.5) circle (3pt) node[align=center,left] (1){1};
\filldraw (0,0.5) circle (3pt) node[align=center,left] (2){2};
\filldraw [color=red] ({sqrt(3)/2},0) circle (3pt) node[align=center,right] (3){3};
\node at (0.5,-2.25) {$\begin{aligned}
&XB_G(\bm x; t)=\\&\phantom{+} (t+1)^3e_{111}\\
&- 3t(t+1)^2e_{21}\\
&+ 3t^2(t+3)e_3
\end{aligned}$};

\node at (3,0) {$H=$};
% Path on 2 vertices (P2)
\draw (3.75,0)--(4.75,0);
\filldraw [color=red] (3.75,0) circle (3pt) node[align=center,below] (1){1};
\filldraw (4.75,0) circle (3pt) node[align=center,below] (2){2};
\node at (4.25,-1.9) {$\begin{aligned}
&XB_H(\bm x; t) =\\&\phantom{-}(t+1)e_{11}\\
&- 2te_2
\end{aligned}$};

\node at (7.5,0) {$G+H=$};
% G and H glued at vertex 3/1
\draw (8.75,-0.5)--(8.75,0.5)--({8.75+sqrt(3)/2},0)--(8.75,-0.5);
\draw ({8.75+sqrt(3)/2},0)--({9.75+sqrt(3)/2},0);
\filldraw (8.75,-0.5) circle (3pt) node[align=center,left] (1){1};
\filldraw (8.75,0.5) circle (3pt) node[align=center,left] (2){2};
\filldraw [color=red] ({8.75+sqrt(3)/2},0) circle (3pt) node[align=center,below] (3){3};
\filldraw ({9.75+sqrt(3)/2},0) circle (3pt) node[align=center,below] (4){4};
\node at (10.25,-2.6) {$\begin{aligned}
XB_{G+H}(\bm x; t) &= (t+1)^4 e_{1111}\\
&- t(t+1)(4t^2+11t+8)e_{211}\\
&+ 2t^2(t+1)(t+2)e_{22}\\
&+ t^2(4t^2+15t+15)e_{31}\\
&- 4t^3(t+3)e_4
\end{aligned}$};
\end{tikzpicture}$$
\caption{\label{fig:glue} Graphs $G$ and $H$ glued at a single vertex.}
\end{figure}

\subsection{Tutte symmetric function}
A \emph{coloring} of $G=([n], E)$ is a function $\kappa:[n] \to \mathbb{P}=\{1,2,3,\ldots\}$. A coloring $\kappa$ is a \emph{proper coloring} if $\kappa(u)\neq \kappa(v)$ whenever $\{u,v\}\in E$. The \emph{chromatic symmetric function} of $G$ is \cite[Definition~2.1]{chromsym}
\begin{equation*}
X_G(\bm x)=\sum_{\kappa\text{ proper coloring }}x_{\kappa(1)}\cdots x_{\kappa(n)}.
\end{equation*}
A generalization is the \emph{Tutte symmetric function} of $G$, defined as \cite[Definition~3.1]{graphrelated}
\begin{equation}
    XB_G(\boldsymbol x; t)=\sum_{\kappa\text{ coloring}}(1+t)^{e(\kappa)}x_{\kappa(1)}\cdots x_{\kappa(n)},
\end{equation}
where we sum over all colorings $\kappa: [n]\rightarrow \mathbb{P}$, letting $e(\kappa)$ be the number of edges $\{u, v\}$ where $\kappa(u)=\kappa(v)$. Note $XB_G(\bm x; -1)=X_G(\bm x)$.

We will write the Tutte symmetric function in the \emph{$e$-basis}. Given a partition $\lambda=(\lambda_1,\ldots, \lambda_\ell)$ of decreasing positive integers, define
\[e_\lambda = e_{\lambda_1}\cdots e_{\lambda_\ell}, \text{ where } e_{n}=\sum_{i_1<\cdots < i_n \in \mathbb{P}}x_{i_1}\cdots x_{i_n}.\]
For brevity, we use juxtaposition to denote $\lambda$ consisting of single digits, meaning $e_{421}=e_{(4,2,1)}$.
\begin{example}
    Figure~\ref{fig:glue} shows an example of gluing graphs together. To calculate $XB_H(\bm x; t)$ for $H=P_2$, we either color vertices 1 and 2 with different colors or the same color, getting
    \[XB_{H}(\bm x; t) = 2\cdot\sum_{i_1<i_2\in\mathbb{P}}x_{i_1}x_{i_2}+\sum_{i\in\mathbb{P}}(1+t)x_i^2=2e_2+(1+t)(e_{11}-2e_2),\]
    which simplifies to the value in the figure.
\end{example}

\subsection{Subgraph triples}

We will now provide a signed combinatorial formula for the $e$-expansion of $XB_G(\boldsymbol x)$ using subgraph triples, similar to the definition in \cite[Definition~5.12]{gluesinglevertex}.

\begin{definition}
A \emph{component triple of $G$} is an object $\mathcal C=(C,\alpha,r)$ consisting of the following data.
\begin{itemize}
\item $C$ is a connected subgraph of $G$.
\item $\alpha$ is a composition of size $|C|$.
\item $r$ is an integer with $1\leq r\leq\alpha_1$, the first part of $\alpha$.
\end{itemize}
A \emph{subgraph triple of $G$} is a sequence $\mathcal S=(\mathcal C_1=(C_1,\alpha^{(1)},r_1),\ldots,\mathcal C_m=(C_m,\alpha^{(m)},r_m))$ of component triples where $S=C_1\sqcup\cdots\sqcup C_m$ is a spanning subgraph of $G$. Let $\text{ST}(G)$ denote the set of subgraph triples of $G$. The \emph{type} of $\mathcal S$ is the partition
\begin{equation*}
\text{type}(\mathcal S)=\text{sort}(\alpha^{(1)}\cdots\alpha^{(m)}),
\end{equation*}
given by concatenating the compositions and sorting to make a partition. The \emph{sign} of $\mathcal S$ is the integer
\begin{equation*}
\text{sign}(\mathcal S)=(-1)^{n-\sum_{i=1}^m(\ell(\alpha^{(i)}))}=(-1)^{n-\ell(\text{type}(\mathcal S))}.
\end{equation*}
We will often be interested in $\alpha^{(1)}_1$, the first part of the composition associated to the component containing vertex $1$. We define the \emph{reduced type} of $\mathcal S$, denoted $\text{type}'(\mathcal S)$, to be the partition $\text{type}(\mathcal S)$ with an instance of $\alpha^{(1)}_1$ removed, meaning
\[\type'(\strip)=\sort(\alpha^{(1)} \setminus \alpha_1^{(1)}\cdot \alpha^{(2)}\cdots \alpha^{(m)}).\]
\end{definition}
Now we can use subgraph triples to calculate the Tutte symmetric function.

\begin{theorem}
The Tutte symmetric function of $G$ satisfies
\begin{equation*}
XB_G(\bm x; t)=\sum_{\mathcal S\in\text{ST}(G)}\text{sign}(\mathcal S) \cdot t^{E(\strip)} \cdot e_{\text{type}(\mathcal S)},
\end{equation*}
where $E(\strip)=|E(S)|$ counts the number of edges in all components of $\strip$.
\end{theorem}
\begin{proof}
The Tutte symmetric function can be calculated as \cite[Theorem~3.2]{graphrelated}
\begin{equation*}
XB_G(\bm x; t) = \sum_{S\subseteq E(G)}t^{|S|}p_{\lambda(S)},
\end{equation*}
where $\lambda(S)$ is the partition formed by the sizes of the connected components in $S$. Using the change of basis formula
\[p_{n}=\sum_{\alpha \models n}(-1)^{n-\ell(\alpha)}\alpha_1e_{\sort(\alpha)},\]
we can rewrite the expression with a product over all connected components $C\in S$:
\begin{align*}
    XB_G(\bm x; t) &= \sum_{S\subseteq E(G)}t^{|S|}\prod_{C\in S}\sum_{\alpha \models |C|}(-1)^{|C|-\ell(\alpha)}\alpha_1 e_{\sort(\alpha)}\\
    &= \sum_{\substack{S\subseteq E(G)\\S=C_1\cup\cdots\cup C_m}}\; \sum_{\substack{\alpha^{(i)}\models |C_i|\\r_i\in\{1,\ldots, \alpha_1^{(i)}\}}}t^{|S|}(-1)^{n-\ell(\alpha^{(1)})-\cdots-\ell(\alpha^{(m)})}e_{\sort(\alpha^{(1)})}\cdots e_{\sort(\alpha^{(m)})}\\
    &=\sum_{\strip\in\ST(G)}\sign(\strip)\cdot t^{E(\strip)}\cdot e_{\type(\strip)}.
\end{align*}
\end{proof}

\begin{figure}
\begin{tikzpicture}
\tikzmath{\h1=0;\v1=0;\d1=3;};
\node[align=center] at ({\h1-2}, {\v1-1.25}) {$XB_{P_2}(\bm x; t)$};
\node[align=center] at ({\h1-0.4}, {\v1-1.25}) {$=$};
\vertex{1}[(\h1,\v1)]{(0,0)};
\vertex{2}[(\h1,\v1)]{(1,0)};
\node[font=\tiny,align=center] at ({\h1-0.2}, {\v1-0.5}) {$\alpha=(1)$ \\ $r=1$};
\node[font=\tiny,align=center] at ({\h1+1.2}, {\v1-0.5}) {$\alpha=(1)$ \\ $r=1$};
\node[align=center] at ({\h1+0.5}, {\v1-1.25}) {$e_{11}$};
\node[align=center] at ({\h1+0.5+(\d1/2)}, {\v1-1.25}) {$+$};
\tikzmath{\h1=\h1+\d1;};
\vertex{1}[(\h1,\v1)]{(0,0)};
\vertex{2}[(\h1,\v1)]{(1,0)};
\node[font=\tiny,align=center] at ({\h1+0.5}, {\v1-0.5}) {$ \alpha=(1,1)$ \\ $r=1$};
\node[align=center] at ({\h1+0.5}, {\v1-1.25}) {$te_{11}$};
\node[align=center] at ({\h1+0.5+(\d1/2)}, {\v1-1.25}) {$+$};
\draw (v1)--(v2);
\tikzmath{\h1=\h1+\d1;};
\vertex{1}[(\h1,\v1)]{(0,0)};
\vertex{2}[(\h1,\v1)]{(1,0)};
\node[font=\tiny,align=center] at ({\h1+0.5}, {\v1-0.5}) {$ \alpha=(2)$ \\ $r=1$};
\node[align=center] at ({\h1+0.5}, {\v1-1.25}) {$-te_{2}$};
\node[align=center] at ({\h1+0.5+(\d1/2)}, {\v1-1.25}) {$+$};
\draw (v1)--(v2);
\tikzmath{\h1=\h1+\d1;};
\vertex{1}[(\h1,\v1)]{(0,0)};
\vertex{2}[(\h1,\v1)]{(1,0)};
\draw (v1)--(v2);
\node[font=\tiny,align=center] at ({\h1+0.5}, {\v1-0.5}) {$ \alpha=(2)$ \\ $r=2$};
\node[align=center] at ({\h1+0.5}, {\v1-1.25}) {$-te_{2}$};
\end{tikzpicture}

\caption{\label{fig:p2_triple} The subgraph triples used to calculate $XB_{P_2}(\bm x; t)$.}
\end{figure}
\begin{example}
    Figure~\ref{fig:p2_triple} shows how to calculate $XB_{P_2}(\bm x; t)$ using subgraph triples.
\end{example}

\section{Tutte Symmetric Matrix}\label{section:tsm}
In \cite[Section~3]{gluesinglevertex}, the authors defined a matrix which allowed $X_{G+H}(\bm x)$ to be calculated from information about $X_G(\bm x)$ and $X_H(\bm x)$. In this section, we generalize this to Tutte symmetric functions.

\begin{definition}\label{def:fmat:fti}
Let $G=([n],E)$ be a graph and consider the subset of $\text{ST}(G+P_j)$
\begin{equation*}
\text{ST}^{(i)}(G+P_j)=\{\strip\in\text{ST}(G+P_j): \ \alpha_1=i, \ r=1, \ \{n,n+1,\ldots,n+j-1\}\subseteq C', \ \alpha'_\ell\geq j\},
\end{equation*}
where $\ctrip=(C,\alpha,r)$ is the component triple of $\strip$ with $1\in C$ and $\ctrip'=(C',\alpha',r')$ is the component triple of $\strip$ with $n\in C'$. In other words, the component containing vertex $1$ must have $r=1$ and a composition with first part $i$, and the component containing vertex $n$ must contain the entire path $P_j$ and its composition must have last part at least $j$. Note that we could have $\ctrip=\ctrip'$. \end{definition}
\begin{definition}\label{def:ftm:Fgij} We define the infinite matrix $M_G$ by
\begin{equation*}
(M_G)_{i,j}=\sum_{\mathcal S\in\text{ST}^{(i)}(G+P_j)}\text{sign}(\strip)\cdot (-1)^{1-j} \cdot t^{E(\strip)+1-j} \cdot e_{\text{type}'(\strip)}.
\end{equation*}

\end{definition}
Although $M_G$ is an infinite matrix, because $G+P_j$ has exactly $n+j-1$ vertices, every column contains a finite number of non-zero entries.\begin{proposition}\label{prop:ftmatrixzeroes}
Let $G=([n],E)$ be a graph. If $i\geq n+j$, or if $j\geq i\geq n\geq 2$, then $\text{ST}^{(i)}(G+P_j)$ is empty, so $(M_G)_{i,j}=0$.
\end{proposition}
\begin{proof}
    If $i \geq n + j$, then $\strip\in \ST^{(i)}(G+P_j)$ has component triple $\ctrip$ with $\alpha_1\geq n+j > |G+P_j|$, which is impossible since the sum of the sizes of the compositions equals $|G+P_j|$. If $j\geq i \geq n \geq 2$, then $\strip\in\ST^{(i)}(G+P_j)$ has $\ctrip$ with $\alpha_1\geq n$. This means vertices 1 and $n$ must be in the same component (so that $\ctrip$ includes all $j$ vertices of $P_j$), meaning $\ctrip = \ctrip'$ and $\alpha_\ell \geq j$. Then, the size of this composition is at least $i+j > n + j - 1 = |G+P_j|$, a contradiction.
\end{proof}
\begin{proposition}\label{prop:xmatrix}
Given $M_G$, we can recover $XB_G(\bm x; t)$ as
\begin{equation}\label{eq:xmatrix}
XB_G(\bm x; t)=\vec v M_G\vec w^T,
\end{equation}
where $\vec{v}, \vec{w}$ are the infinite row vectors \begin{equation*}\vec v=\left[\begin{matrix}e_1&2e_2&3e_3&\cdots\end{matrix}\right]\text{ and }\vec w=\left[\begin{matrix}1&0&0&\cdots\end{matrix}\right].\end{equation*} 
Note that because the columns of $M_G$ have finitely many nonzero entries by Proposition~\ref{prop:ftmatrixzeroes}, the matrix multiplications are valid.
\end{proposition}
\begin{proof}
We have
\begin{equation*}
XB_G(\bm x; t)=\sum_{i\geq 1}ie_i\left(\sum_{\strip\in\text{ST}^{(i)}(G+P_1)}\text{sign}(\strip)\cdot t^{E(\strip)}\cdot e_{\text{type}'(\strip)}\right)=\sum_{i\geq 1}ie_i(M_G)_{i,1}=\vec{v}M_G\vec w^T,
\end{equation*}
where the factor of $i$ in $ie_i$ accounts for the possible choices of $1\leq r_i\leq\alpha^{(1)}_1=i$.
\end{proof}

We now calculate a few examples of $M_G$.

\begin{proposition}\label{prop:ftp1}
    For the graph $P_1$ with one vertex, $M_{P_1}$ is the infinite identity matrix $I$.
\end{proposition}
\begin{proof}
    We are considering subgraph triples of $P_1+P_j=P_j$. Since $n=1$, then $\ctrip=\ctrip'$, meaning there is only 1 subgraph triple $\strip$ which resides in $\ST^{(j)}(G+P_j)$ and consists of the single component triple $\ctrip=(P_j,(j),1)$. When $i\neq j$, since $\ST^{(i)}(G+P_j)=\emptyset$ then $(M_{P_1})_{i,j}=0$. Similarly,
    \[(M_{P_1})_{j,j} = (-1)^{j-1}\cdot (-1)^{1-j}\cdot t^{j-1+1-j}\cdot 1 = 1,\]
    so $M_{P_1}$ is the infinite identity matrix.
\end{proof}

\begin{proposition}\label{prop:ftp2}
For the two-vertex path $P_2$, we have
\begin{equation*}
(M_{P_2})_{i,j}=\begin{cases} (t+j)e_j&\text{ if }i=1,\\
-t&\text{ if }i=j+1,\\
0&\text{ otherwise}.\end{cases}
\hspace{20pt}
M_{P_2}=\left[\begin{matrix}
(t+1)e_1&(t+2)e_2&(t+3)e_3&\cdots\\
-t&0&0&\cdots\\
0&-t&0&\cdots\\
0&0&-t&\cdots\\
%0&0&0&0&\cdots\\
\vdots&\vdots&\vdots&\ddots\\
\end{matrix}\right].
\end{equation*}
\end{proposition}

\begin{proof}
We are considering subgraph triples of $P_2+P_j=P_{j+1}$, shown in Figure~\ref{fig:ftp2example}. If the vertices $1$ and $2$ are in different component triples $\ctrip=(C,\alpha,r)$ and $\ctrip'=(C',\alpha',r')$, then we must have $\alpha=(1)$ and $r=1$. Because $\alpha'_\ell\geq j$, we must have $\alpha'=(j)$ and there are $j$ possibilities corresponding to the choices of $1\leq r'\leq j$. If the vertices $1$ and $2$ are in the same component triple $\ctrip=(C,\alpha,r)$, then $C$ must be the entire path $P_{j+1}$. Because $\alpha_\ell\geq j$, the only possibilities are $\alpha=(j+1)$ and $\alpha=(1,j)$, and we must have $r=1$.
\end{proof}
\begin{figure}
$$\begin{tikzpicture}
\node at (7.9,-1) [color=teal] {$\alpha=j+1$};
\node at (7.7,-2) [color=teal] {$\alpha=1 \ j$};
\node at (7.5,0) [color=magenta] {$\alpha'=j$};
\node at (9.5,0) [color=magenta] {$r'=1,\ldots,j$};
\node at (-1,0) [color=cyan] {$\alpha=1$};
\node at (12,-1) {$-t$};
\node at (12,-2) {$te_j$};
\node at (12,0) {$je_j$};
\filldraw [color=teal](0,-1) circle (3pt) node[align=center,below,color=teal] (1){1};
\filldraw [color=teal](1,-1) circle (3pt) node[align=center,below,color=teal] (2){2};
\filldraw [color=teal](2,-1) circle (3pt) node[align=center,below,color=teal] (3){3};
\filldraw [color=teal](3,-1) circle (3pt) node[align=center,below,color=teal] (){};
\filldraw [color=teal](4,-1) circle (3pt) node[align=center,below,color=teal] (){};
\filldraw [color=teal](5,-1) circle (3pt) node[align=center,below,color=teal] ($j$){$j$};
\filldraw [color=teal](6,-1) circle (3pt) node[align=center,below,color=teal] ($j+1$){$j+1$};
\draw [color=teal](0,-1)--(2,-1) (5,-1)--(6,-1);
\draw [color=teal, dashed] (2,-1)--(5,-1);

\filldraw [color=teal](0,-2) circle (3pt) node[align=center,below,color=teal] (1){1};
\filldraw [color=teal](1,-2) circle (3pt) node[align=center,below,color=teal] (2){2};
\filldraw [color=teal](2,-2) circle (3pt) node[align=center,below,color=teal] (3){3};
\filldraw [color=teal](3,-2) circle (3pt) node[align=center,below,color=teal] (){};
\filldraw [color=teal](4,-2) circle (3pt) node[align=center,below,color=teal] (){};
\filldraw [color=teal](5,-2) circle (3pt) node[align=center,below,color=teal] ($j$){$j$};
\filldraw [color=teal](6,-2) circle (3pt) node[align=center,below,color=teal] ($j+1$){$j+1$};
\draw [color=teal](0,-2)--(2,-2) (5,-2)--(6,-2);
\draw [color=teal, dashed] (2,-2)--(5,-2);

\filldraw [color=cyan](0,0) circle (3pt) node[align=center,below,color=cyan] (1){1};
\filldraw [color=magenta](1,0) circle (3pt) node[align=center,below,color=magenta] (2){2};
\filldraw [color=magenta](2,0) circle (3pt) node[align=center,below,color=magenta] (3){3};
\filldraw [color=magenta](3,0) circle (3pt) node[align=center,below,color=magenta] (){};
\filldraw [color=magenta](4,0) circle (3pt) node[align=center,below,color=magenta] (){};
\filldraw [color=magenta](5,0) circle (3pt) node[align=center,below,color=magenta] ($j$){$j$};
\filldraw [color=magenta](6,0) circle (3pt) node[align=center,below,color=magenta] ($j+1$){$j+1$};
\draw [color=magenta] (1,0)--(2,0) (5,0)--(6,0);
\draw [color=magenta, dashed] (2,0)--(5,0);

\end{tikzpicture}
$$
\caption{\label{fig:ftp2example} The subgraph triples used to calculate $M_{P_2}$.}
\end{figure}
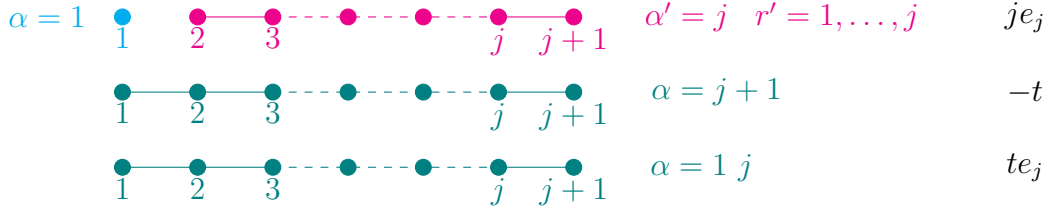

It turns out that the operation of gluing graphs corresponds to multiplying matrices.

\begin{proposition}\label{prop:ftmult}
For graphs $G$ and $H$, we have
\begin{equation*}
M_{G+H}=M_GM_H.
\end{equation*}
\end{proposition}
The proof is nearly identical to the proof in \cite[Section~3]{gluesinglevertex} for a similarly defined matrix for $X_G(\bm x)$, except we must now keep track of the power of $t$ as well.

\begin{proof}
Given graphs $G, H$ with $n=|G|$ and $n'=|H|$, we construct a bijection
\begin{equation}\label{eqn:attachatone_signature}\varphi:\bigsqcup_{k\geq 1}\ST^{(i)}(G+P_k)\times\ST^{(k)}(H+P_j) \to \ST^{(i)}(G+H+P_j)
\end{equation}
%\begin{equation}\label{eqn:break_signature}\text{break}:\ST^{(i)}(G+H+P_j)\to\bigsqcup_{k\geq 1}\ST^{(i)}(G+P_k)\times\ST^{(k)}(H+P_j)\end{equation}
as follows: given
\[\mathcal{S}=\{\ctrip, \ctrip_1,\ldots,\ctrip_m\}\in\ST^{(i)}(G+P_k) \text{ and } \mathcal{S}'=\{\ctrip',\ctrip'_1,\ldots,\ctrip'_{m'}\}\in\ST^{(k)}(H+P_j),\] let $\ctrip=(C,\alpha,r)$ be the component triple with $n\in C$ and $\ctrip'=(C', \alpha', 1)\in\mathcal{S}'$ be the component triple with $1 \in C'$. Then, after adding $(n-1)$ to all vertices in $\strip'$ so they lie in $[n, n+n'+k-2]$, we define
\[\varphi(\strip, \strip')=\big\{\big((C\setminus P_k) \cup C', \alpha_1\cdots \alpha_{\ell}\cdot \alpha'_2\cdots\alpha'_\ell, r\big), \ctrip_1,\ldots,\ctrip_m,\ctrip'_1,\ldots,\ctrip'_{m'}\big\},\]
where $((C \setminus P_k) \cup C')$ is the component obtained by removing the path $P_k$ from component $C$ and attaching component $C'$, shown in Figure~\ref{fig:ftvarphiexample}. This results in a valid component triple since we remove $k$ vertices and remove $\alpha'_1=k$ from the composition, so the new composition indeed has size $|(C\setminus P_k)\cup C'|$.

The inverse is defined as follows: given $\strip \in \ST^{(i)}(G+H+P_j)$, let $\ctrip = (C, \alpha, r) \in \strip$ be the component with vertex $n\in C$. Denote $C\vert_G$ and $C\vert_H$ to be the restriction of $C$ to the vertices of $G$ and $H$ respectively, so $|C|=|C\vert_G|+|C\vert_H|-1$. Since the sum of the components in $\alpha$ is $|C| \geq |C\vert_G|$, there exists a minimal $s$ where $\alpha_1+\cdots+\alpha_s\geq |C\vert_G|$. Let
\[k=\alpha_1+\cdots+\alpha_s-|C\vert_G| + 1,\]
and define
\[\ctrip\!\vert_G = (C\vert_G \cup P_k, \alpha_1\cdots \alpha_s, r) \text{ and } \ctrip\!\vert_H = (C\vert_H, k\ \alpha_{s+1}\cdots \alpha_{\ell}, 1).\]
Define $\mathcal{S}\vert_G$ to be the subgraph triple containing $\mathcal{C}\vert_G$ and all component triples of $\strip$ whose components contain vertices less than $n$. Similarly, $\mathcal{S}\vert_H$ contains $\mathcal{C}\vert_H$ and all component triples with vertices greater than $n$ (after subtracting $(n-1)$ from all vertices so they lie in $[n'+k-1]$). Then, $\varphi^{-1}(\strip)=(\mathcal{S}\vert_G, \mathcal{S}\vert_H)$.

Note the inverse is well-defined; apart from the newly created $\mathcal{C}\vert_H$, the first part of the composition of every component triple is preserved and hence $\mathcal{S}\vert_G\in \ST^{(i)}(G+P_k)$. Similarly, the first part of the composition of $\mathcal{C}\vert_H$ is $k$ and hence $\mathcal{S}\vert_H \in \ST^{(k)}(H+P_j)$. Next, given $\strip=\varphi(\mathcal{S}\vert_G, \mathcal{S}\vert_H)$, we have
\begin{gather*}
\sign(\strip)=\sign(\mathcal{S}\vert_G)\sign(\mathcal{S}\vert_H),\\
e_{\type'(\mathcal{S})}=e_{\type'(\mathcal{S}\vert_G)}e_{\type'(\mathcal{S}\vert_H)},\\
E(\mathcal{S})=E(\mathcal{S}\vert_G)+E(\mathcal{S}\vert_H)-k+1.
\end{gather*}

Then we have, as desired, that
\begin{align*}
(M_{G+H})_{i,j}&=\sum_{\strip\in\ST^{(i)}(G+H+P_j)}\sign(\strip)\cdot t^{E(\strip)-j+1}\cdot e_{\type'(\strip)}\\
&=\sum_{k\geq 1}\left(\sum_{\strip\!\vert_G\in\ST^{(i)}(G+P_k)}\sign(\strip\!\vert_G)\cdot t^{E(\strip\!\vert_G)-k+1}\cdot e_{\type'(\strip\!\vert_G)}\right)\cdot\\
&\hspace{1.13cm}\left(\sum_{\strip\!\vert_H\in\ST^{(k)}(H+P_j)}\sign(\strip\!\vert_H)\cdot t^{E(\strip\!\vert_H)-j+1}\cdot e_{\type'(\strip\!\vert_H)}\right)\\
&=\sum_{k\geq 1}(M_G)_{i,k}(M_H)_{k,j}=(M_GM_H)_{i,j}.
\end{align*}
\end{proof}

    \begin{figure}
        \begin{tikzpicture}
        \tikzmath{\h1=8.5;\v1=-1.5;\h2=-1;};

\vertex[teal][-3pt][below]{1}[({\h1+\h2},\v1)]{(0,0)};
\vertex[magenta]{2}[({\h1+\h2},\v1)]{(0,1)};
\vertex[teal]{3}[({\h1+\h2},\v1)]{(1,1)};
\vertex[magenta][-3pt][below]{4}[({\h1+\h2},\v1)]{(1,0)};
\vertex[red]{5}[({\h1+\h2},\v1)]{({1+sqrt(3)/2}, 0.5)};
\vertex[teal]{6}[({\h1+\h2},\v1)]{({1+sqrt(3)}, 1)};
\vertex[teal][-3pt][below]{7}[({\h1+\h2},\v1)]{({1+sqrt(3)}, 0)};
\vertex[teal][-3pt][below]{8}[({\h1+\h2},\v1)]{({2+sqrt(3)}, 0)};

\draw[color=teal] (v1)--(v3)--(v5)--(v6)--(v7)--(v5) (v7)--(v8);
\draw [color=magenta] (v2)--(v4);
\draw[dashed, opacity=0.4] (v1)--(v2)--(v3)--(v4)--(v1) (v4)--(v5);

\node[align=center] at (\h1+1,-3) {$\strip\in \ST^{(1)}(G+P_2)$,\\$\textcolor{teal}{\ctrip=((C\setminus P_2)\cup C',\alpha=132, r=1)}$,\\
$s=2, k=\alpha_1+\alpha_2-|\mathcal{C}\vert_G|+1=2$.};

\tikzmath{\arrplace=(\h1+1)/2;\h1=-4;\h2=1;\v1=-1.5;\v2=-1;}

\draw [arrows = {-Stealth[scale=2]}] (\arrplace-1.25,-1)--(\arrplace+1.75,-1);
\draw [arrows = {-Stealth[scale=2]}] (\arrplace+1.75,-1)--(\arrplace-1.25,-1);

\vertex[teal][-3pt][below]{1}[(\h1,\v1)]{(0,0)};
\vertex[magenta]{2}[(\h1,\v1)]{(0,1)};
\vertex[teal]{3}[(\h1,\v1)]{(1,1)};
\vertex[magenta][-3pt][below]{4}[(\h1,\v1)]{(1,0)};
\vertex[red]{5}[(\h1,\v1)]{({1+sqrt(3)/2}, 0.5)};
\vertex[teal]{6}[(\h1,\v1)]{({2+sqrt(3)/2}, 0.5)};

\draw[color=teal] (v1)--(v3)--(v5)--(v6);
\draw [color=magenta] (v2)--(v4);
\draw[dashed, opacity=0.4] (v1)--(v2)--(v3)--(v4)--(v1) (v4)--(v5);

\vertex[red]{1}[(\h2,\v2)]{(0, 0)};
\vertex[teal]{2}[(\h2,\v2)]{({sqrt(3)/2}, 0.5)};
\vertex[teal][-3pt][below]{3}[(\h2,\v2)]{({sqrt(3)/2}, -0.5)};
\vertex[teal][-3pt][below]{4}[(\h2,\v2)]{({1+sqrt(3)/2}, -0.5)};
\draw[color=teal] (v1)--(v2)--(v3)--(v1) (v3)--(v4);

\node[align=center, font=\Huge] at ({(\h1+\h2+3)/2}, \v2) {$+$};
\node[align=center] at ({\h1+1}, {\v2-2}) {$\strip\in \ST^{(1)}(G+P_2)$,\\
\textcolor{teal}{$\ctrip = (C, \alpha=13, r=1)$}.\\
};
\node[align=center] at ({\h2+1}, {\v2-2}) {$\strip\in \ST^{(2)}(G+P_2)$,\\
\textcolor{teal}{$\ctrip' = (C', \alpha=22, r=1)$}.\\
};

        \end{tikzpicture}
    \caption{\label{fig:ftvarphiexample} Example of map $\varphi$ with $i=1, j=2, k=2$.}
    \end{figure}
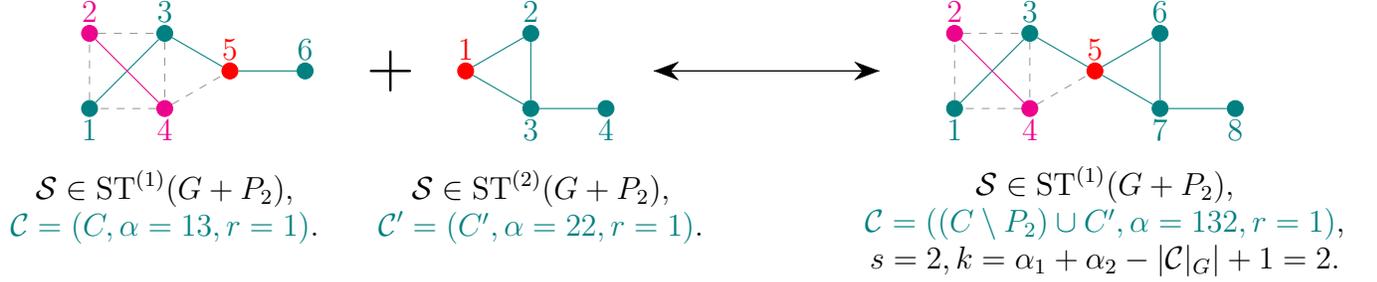

Using this property, we can calculate $M_{P_n}$.
\begin{proposition}\label{prop:ftpn}
For the path $P_n$, we have
\begin{equation}\label{eqn:ftm:fpn}
(M_{P_n})_{i,j}=\sum_{\substack{\alpha\models n+j-i-1,\\ \alpha_\ell\geq j}}(\alpha_1+t)\cdots(\alpha_\ell+t)\cdot (-t)^{n-\ell(\alpha)-1}e_{\text{sort}(\alpha)}.
\end{equation}
Note that if $n+j-i-1 = 0$, we include the empty composition in our sum, meaning the empty product $(\alpha_1+t)\cdots(\alpha_{\ell}+t)=1$, so $(M_{P_n})_{n+j-1,j}=(-t)^{n-1}$.
\end{proposition}

\begin{proof}
For the base case $n\leq 2$, the result follows from Proposition~\ref{prop:ftp1} and Proposition~\ref{prop:ftp2}. For $n\geq 3$, because $P_n=P_{n-1}+P_2$, we use induction and Proposition~\ref{prop:ftmult}. Defining the polynomial $w_\alpha=(\alpha_1+t)\cdots(\alpha_\ell+t)\cdot (-t)^{n-\ell(\alpha)-1}$, we have
\begin{align*}
(M_{P_n})_{i,j}&=(M_{P_{n-1}}M_{P_2})_{i,j}=(M_{P_{n-1}})_{i,1}\cdot (t+j)e_j+(M_{P_{n-1}})_{i,j+1}\cdot (-t)\\
&=\sum_{\alpha\models n-1-i}(-t)^{-1}w_\alpha e_{\text{sort}(\alpha)}(j+t)e_j+\sum_{\alpha\models n+j-i-1, \ \alpha_\ell\geq j+1}(-t)^{-1}w_\alpha e_{\sort(\alpha)}(-t)\\
&=\sum_{\substack{\alpha\models n+j-i-1\\\alpha_\ell=j}} w_\alpha e_{\text{sort}(\alpha)}+\sum_{\substack{\alpha\models n+j-i-1\\\alpha_\ell\geq j+1}}w_\alpha e_{\text{sort}(\alpha)}=\sum_{\substack{\alpha\models n+j-i-1\\\alpha_\ell\geq j}} w_\alpha e_{\text{sort}(\alpha)}.
\end{align*}
\end{proof}
\begin{remark} We could also calculate $M_{P_n}$ by enumerating $\ST^{(i)}(P_n+P_j)$ and derive the same result.\end{remark}
\begin{theorem}
    The Tutte symmetric function of $P_n$ is
    \begin{equation}\label{eq:xbpn}XB_{P_n}(\bm x; t) = \sum_{\alpha\models n}\alpha_1(\alpha_2+t)\cdots(\alpha_\ell+t)(-t)^{n-\ell(\alpha)}e_{\sort(\alpha)}.\end{equation}
\end{theorem}
\begin{proof}
    Using Equation~\ref{eq:xmatrix}, we calculate
    \begin{align*}
XB_{P_n}(\bm x; t)&=\vec{v}M_{P_n}\vec{w}^T=\sum_{i}ie_i(M_{P_n})_{i,1}\\
&=\sum_{i}ie_i\sum_{\substack{\alpha\models n-i}}(\alpha_1+t)\cdots(\alpha_{\ell}+t)(-t)^{n-\ell(\alpha)-1}e_{\sort(\alpha)}\\
&=\sum_{\alpha\models n}\alpha_1(\alpha_2+t)\cdots(\alpha_{\ell}+t)(-t)^{n-\ell(\alpha)}e_{\sort(\alpha)},
    \end{align*}
where the final step involves replacing the double sum with a single sum where $\alpha_1=i$.
\end{proof}
\begin{remark}
    Note the formula for the chromatic symmetric function for $P_n$ is
    \[X_{P_n}(\bm x) = \sum_{\alpha\models n}\alpha_1(\alpha_2-1)\cdots(\alpha_{\ell}-1)e_{\sort(\alpha)}.\]
    The Tutte symmetric analog replaces the $-1$ with $t$ and adds additional $-t$ factors.
\end{remark}
Instead of gluing graphs at separate vertices, we can consider gluing graphs at a fixed vertex $x$ of $G$. We define $(G, x)+(H,y)$ by gluing vertices $x$ and $y$ together, with an example shown in Figure~\ref{fig:S_4}. We define the subset of subgraph triples
\begin{equation*}
\ST^{(i)}((G,x)+P_j)=\{\strip\in\ST((G,x)+P_j): \ \alpha_1=i, r=1, P_j\subseteq C, \alpha_\ell\geq j\},
\end{equation*}
where $(G,x)+P_j$ is obtained by attaching the path of length $j$ at vertex $x$ of $G$ and $\mathcal C=(C,\alpha,r)$ is the component triple containing $x$. Let
\begin{equation*}
(M_{(G,x)})_{i,j}=\sum_{\mathcal \strip\in\ST^{(i)}((G,x)+P_j)}\sign(\strip) \cdot (-1)^{1-j} \cdot t^{E(\strip)+1-j} \cdot e_{\type'(\strip)},
\end{equation*}
then the same arguments as before give us that
\begin{equation*}
XB_G(\bm x; t)=\vec vM_{(G,x)}\vec w^T\text{ and }M_{(G,x)+(H,y)}=M_{(G,x)}M_{(H,y)}.
\end{equation*}
\begin{figure}
    \begin{tikzpicture}
        \vertexs[red][-3pt][30]{x}{(0,0)};
        \node[red] at (0.2,0.2) {$x$};
        \vertexs{2}{(0,1)};
        \vertexs{3}{(210:1)};
        \vertexs{4}{(330:1)};
        \draw (v3)--(vx)--(v2) (vx)--(v4);
    \end{tikzpicture}
    \caption{\label{fig:S_4} The graph $S_4=K_{1,3}$ obtained by gluing three $P_2$ graphs at a common vertex $x$.}
\end{figure}
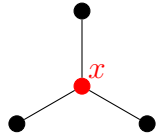
\begin{proposition}
    Let $S_n=K_{1,n-1}$ be the star graph with $n$ vertices and let $x$ be the center vertex, then
    \begin{equation}\label{eq:star_final}(M_{(S_n,x)})_{i,j}=\sum_{\substack{\alpha\models n+j-i, \alpha_1 \leq n\\
        \len(\alpha) \geq 1\\ \text{if $\len(\alpha)\geq 2$ then $\alpha_{\ell}\geq j$}}}\binom{n-1}{\alpha_1-1}(-t)^{n-\alpha_1}(e_1)^{\alpha_1-1}(-1)^{\ell(\alpha)-1}e_{\sort(\alpha \setminus \alpha_1)},\end{equation}
    The restriction $\len(\alpha)\geq 1$ is because we do not allow the empty composition in this sum.
\end{proposition}
\begin{proof}
    We will calculate $(M_{(S_n,x)})_{i,j}$ by enumerating over $\ST^{(i)}((S_n,x)+P_j)$. For any $\strip$, let $\ctrip = (C, \alpha, r)$ be the component triple with $x\in C$. Let
    \[A_k = \left\{\strip \in \ST^{(i)}((S_n, x)+P_j) : |C| - j = k\right\},\]
    meaning $A_k$ is the set of subgraph triples where $k$ of the $n-1$ non-center vertices are connected to the center. There are $\binom{n-1}{k}$ possible subgraphs so
    \[\sum_{\strip\in A_k} \sign(\strip)\cdot (-1)^{1-j} \cdot t^{E(\strip)+1-j}\cdot e_{\type'(\strip)}=\binom{n-1}{k} \cdot (-t)^{k}\cdot (e_1)^{n-k-1}\cdot \sum_{\substack{\alpha\models k+j-i\\\alpha_\ell\geq j}}(-1)^{\ell(\alpha)}e_{\sort(\alpha)}.\]
    Since $k\in [0,n-1]$, summing over $k$ gives
    \begin{equation}\label{eq:star_doublesum}(M_{(S_n,x)})_{i,j} = \sum_{k=0}^{n-1}\binom{n-1}{k}(-t)^{k}(e_1)^{n-k-1}\sum_{\substack{\alpha\models k+j-i\\\alpha_\ell\geq j}}(-1)^{\ell(\alpha)}e_{\sort(\alpha)}.\end{equation}
    Note in the sum above, if $k+j-i=0$, then the empty composition is included in the sum (since $\alpha_\ell$ doesn't exist, we say it vacuously satisfies the condition). To simplify, we replace $k$ with $n-1-k$ and then turn the double sum into a single sum over compositions $\alpha \models n+j-i$ where $k = \alpha_1-1$:
    \begin{align*}
        (M_{(S_n,x)})_{i,j}&=\sum_{k=0}^{n-1}\binom{n-1}{k}(-t)^{n-k-1}(e_1)^k\sum_{\substack{\alpha\models n+j-k-i-1\\\alpha_{\ell}\geq j}}(-1)^{\ell(\alpha)}e_{\sort(\alpha)}\\
        &=\sum_{\substack{\alpha\models n+j-i, \alpha_1 \leq n\\
        \len(\alpha) \geq 1\\ \text{if $\len(\alpha)\geq 2$ then $\alpha_{\ell}\geq j$}}}\binom{n-1}{\alpha_1-1}(-t)^{n-\alpha_1}(e_1)^{\alpha_1-1}(-1)^{\ell(\alpha)-1}e_{\sort(\alpha \setminus \alpha_1)}
    \end{align*}
    In the second line, the sum splits into two conditions ($\len(\alpha)=1$ or $\len(\alpha)\geq 2$) because Equation~\ref{eq:star_doublesum} allows empty compositions in the summand. Note $\alpha\setminus \alpha_1$ means remove the first part of the composition.
\end{proof}
\begin{remark}
    Alternatively, one can calculate $M_{(S_n, x)}=(M_{(S_2,x)})^{n-1}$ by first calculating $M_{(S_2, x)}$ and then using matrix multiplication to derive the general formula.
\end{remark}
\begin{theorem}
    The Tutte symmetric function for $S_n$ is
    \[XB_{S_n}(\bm x; t)=\sum_{\substack{\alpha \models n+1\\\alpha_1\leq n}}\alpha_2\binom{n-1}{\alpha_1-1}(-t)^{n-\alpha_1}(e_1)^{\alpha_1-1}(-1)^{\ell(\alpha)}e_{\sort(\alpha\setminus\alpha_1)}.\]
\end{theorem}
\begin{proof}
    Using Equation~\ref{eq:xmatrix}, we have
    \begin{align*}
        XB_{S_n}(\bm x; t) &= \vec{v}M_{S_n}\vec{w}^T=\sum_{i}ie_i(M_{S_n})_{i, 1}\\
        &=\sum_{i} i e_i \sum_{\substack{\alpha \models n+1-i\\\alpha_1\leq n\\\ell(\alpha)\geq 1}}\binom{n-1}{\alpha_1-1}(-t)^{n-\alpha_1}(e_1)^{\alpha_1-1}(-1)^{\ell(\alpha)-1}e_{\sort(\alpha\setminus \alpha_1)}\\
        &=\sum_{\alpha\models n+1, \alpha_1\leq n}\alpha_2\binom{n-1}{\alpha_1-1}(-t)^{n-\alpha_1}(e_1)^{\alpha_1-1}(-1)^{\ell(\alpha)}e_{\sort(\alpha\setminus \alpha_1)},
    \end{align*}
    where the final line involves replacing the double sum with a single sum by inserting $i=\alpha_2$ into the composition.
\end{proof}

\begin{comment}
Instead of connecting graphs at different vertices, we can imagine connecting graphs at the same vertex. Given graphs $G, H$ and vertices $a,b \in V(G)$ and $c,d \in V(H)$, we define
\[(G,a,b) + (H,c,d) = \left((G \sqcup H) / \{b \sim c\}, a, d \right).\]
Informally, we glue vertex $b$ of $G$ to vertex $c$ of $H$, generalizing the previous definition of attaching graphs. In fact, setting $a,b,c,d = 1,n,1,n'$, we get
\[(G, 1, n) + (H, 1, n') = (G + H, 1, n + n' - 1).\]
\end{comment}

\section{Gluing the first and last vertices}
\label{section:traceresult}

In this section, we show that $XB_{G^\circ}(\bm x; t)=\trace(M_G)$, where the graph $G^\circ$ is obtained by gluing vertices $1$ and $n$ of $G$ together.

\begin{definition}
Let $G=([n],E)$ be a graph with $n\geq 2$ vertices. We define the multigraph
\begin{equation*}
G^\circ=([n-1],\{\{i,j\}\in E: \ 1\leq i<j\leq n-1\}\cup\{\{i,1\} : \{i,n\}\in E\}).
\end{equation*}
\end{definition}

Note $G^\circ$ is allowed to have multiple edges between the same vertices, which ensures that $G^\circ$ has the same number of edges as $G$. We now state the main result of this section. Note that by Proposition~\ref{prop:ftmatrixzeroes}, we have $(M_G)_{k,k}=0$ for $k\geq n$, so it makes sense to take the trace of $M_G$.

\begin{theorem}\label{thm:trace:trace}
Let $G=([n],E)$ be a graph with $n\geq 2$ vertices. Then we have
\begin{equation*}
XB_{G^\circ}(\bm x; t)=\trace(M_G)=\sum_{k=1}^{n-1}(M_G)_{k,k}.
\end{equation*}
\end{theorem}
The proof is nearly identical to \cite[Section~5]{gluesinglevertex}, and involves finding a bijection on subgraph triples.

\begin{proof}
Recall $G$ and $G^\circ$ have the same number of edges (as we allow multiple edges in $G^\circ$). Thus, we have a bijection on edge multisets $\varphi : E(G) \to E(G^\circ)$ where for $u < v$, we have
\[\varphi(\{u,v\})=\begin{cases}
\{u,1\},&\text{ if }v=n,\\
\{u,v\},&\text{ otherwise.}
\end{cases}\]
Given subgraph $H=(V(H), E(H))$ of $G$, we define the subgraph $\varphi(H)\subseteq G^\circ$ as
\[\varphi(H)=(V(H)\setminus \{n\}, \varphi(E(H))),\]
and given a subgraph $H^\circ=(V(H^\circ), E(H^\circ))$ of $G^\circ$, we define $\varphi^{-1}(H^\circ)\subseteq G$ as
\begin{equation*}
\varphi^{-1}(H^\circ)=(V^*,\varphi^{-1}(E(H^\circ))),\text{ where }V^*=\begin{cases}
V(H^\circ)\cup\{n\},&\text{ if }1\in V(H^\circ),\\
V(H^\circ),&\text{ otherwise.}
\end{cases}
\end{equation*}

   \begin{figure}
    \begin{tikzpicture}

    %% ── BOTTOM ROW ── G° on the LEFT, G+P3 on the RIGHT ──────────────────

   \tikzmath{\h1=0;}
    \vertex[red]{1}[(\h1, 0)]{(0.5, 0)};
    \vertexs[gray]{2}[(\h1,0)]{(-1, 0)};
    \vertexs[teal]{3}[(\h1,0)]{(-0.5,-0.866)};
    \vertexs[teal]{5}[(\h1,0)]{(-1,-1.732)};
    \vertexs[teal]{6}[(\h1,0)]{(0,-1.732)};
    \vertexs[teal]{7}[(\h1,0)]{(-0.5,-2.598)};
    \vertexs[teal]{8}[(\h1,0)]{(0,-3.464)};
    \vertexs[teal]{9}[(\h1,0)]{(1,-3.464)};
    \vertexs[teal]{10}[(\h1,0)]{(1.5,-2.598)};
    \vertexs[teal]{11}[(\h1,0)]{(1,-1.732)};
    \vertexs[teal]{12}[(\h1,0)]{(2,-1.732)};
    \vertexs[gray]{13}[(\h1,0)]{(2.5,-0.866)};
    \vertexs[gray]{14}[(\h1,0)]{(2,0)};
    \draw[color=teal] (v1)--(v3) (v3)--(v5)--(v7)--(v6)--(v3) (v5)--(v6) (v7)--(v8)--(v9);
    % magenta: vertices 10,11,12,15,16,17,18
    \draw[color=teal] (v10)--(v11) (v10)--(v12)--(v11) (v1)--(v11) (v1)--(v12);
    % teal: vertices 2,13,14
    \draw[color=gray] (v13)--(v14);

    \draw[dashed, opacity=0.4] (v1)--(v2)--(v3) (v2)--(v5) (v9)--(v10) (v11)--(v6) (v12)--(v13) (v1)--(v13) (v1)--(v14);

    \node[align=center] at (\h1+0.75,-4.6) {$\strip\in \ST(G^\circ)$,\\$\textcolor{teal}{\ctrip^\circ=(C^\circ=\varphi(C \cup C'),\alpha=253, r=2)}$\\$s=2, k=\alpha_1+\alpha_2-4+1=4$};

    \tikzmath{\arrplace=-4.5;\varr=-1.732;}
    \draw [arrows = {-Stealth[scale=2]}] (\arrplace-1.75,\varr)--(\arrplace+1.75,\varr);
    \draw [arrows = {-Stealth[scale=2]}] (\arrplace+1.75,\varr)--(\arrplace-1.75,\varr);

    \tikzmath{\h1=-10;}
    \vertex[red]{1}[(\h1, 0)]{(0, 0)};
    \vertexs[gray]{2}[(\h1,0)]{(-1, 0)};
    \vertexs[cyan]{3}[(\h1,0)]{(-0.5,-0.866)};
    \vertexs[cyan]{5}[(\h1,0)]{(-1,-1.732)};
    \vertexs[cyan]{6}[(\h1,0)]{(0,-1.732)};
    \vertexs[cyan]{7}[(\h1,0)]{(-0.5,-2.598)};
    \vertexs[cyan]{8}[(\h1,0)]{(0,-3.464)};
    \vertexs[cyan]{9}[(\h1,0)]{(1,-3.464)};
    \vertexs[magenta]{10}[(\h1,0)]{(1.5,-2.598)};
    \vertexs[magenta]{11}[(\h1,0)]{(1,-1.732)};
    \vertexs[magenta]{12}[(\h1,0)]{(2,-1.732)};
    \vertexs[gray]{13}[(\h1,0)]{(2.5,-0.866)};
    \vertexs[gray]{14}[(\h1,0)]{(2,0)};
    \vertex[red][-3pt][left]{15}[(\h1,0)]{(1,0)};
    \vertex[magenta][-3pt][left]{16}[(\h1,0)]{(1,0.5)};
    \vertex[magenta][-3pt][left]{17}[(\h1,0)]{(1,1)};
    \vertex[magenta][-3pt][left]{18}[(\h1,0)]{(1,1.5)};

    % cyan: vertices 1,3,5,6,7,8,9
    \draw[color=cyan] (v1)--(v3) (v3)--(v5)--(v7)--(v6)--(v3) (v5)--(v6) (v7)--(v8)--(v9);
    % magenta: vertices 10,11,12,15,16,17,18
    \draw[color=magenta] (v10)--(v11) (v10)--(v12)--(v11) (v15)--(v11) (v15)--(v12) (v15)--(v16)--(v17)--(v18);
    % teal: vertices 2,13,14
    \draw[color=gray] (v13)--(v14);

    \draw[dashed, opacity=0.4] (v1)--(v2)--(v3) (v2)--(v5) (v9)--(v10) (v11)--(v6) (v12)--(v13) (v15)--(v13) (v15)--(v14);

    %\draw[dashed, opacity=0.4] (v1)--(v2)--(v3)--(v1) (v2)--(v5) (v3)--(v5)--(v7)--(v6)--(v3) (v5)--(v6) (v7)--(v8)--(v9)--(v10)--(v11)--(v6) (v10)--(v12)--(v11) (v12)--(v13)--(v14) (v15)--(v11) (v15)--(v12) (v15)--(v13) (v15)--(v14) (v15)--(v16)--(v17)--(v18);

    \node[align=center] at (\h1+0.75,-4.6) {$\strip\in \ST^{(4)}(G+P_4)$,\\$\textcolor{cyan}{\ctrip=(C,\alpha=43, r=1)}$\\$\textcolor{magenta}{\ctrip'=(C',\alpha'=25, r'=2)}$\\$k=4$};

%% ── BOTTOM ROW ──
    \tikzmath{\v1=-8;}  % <-- change this value to shift the whole row up/down

    \tikzmath{\h1=0;}
    \vertex[red]{1}[(\h1, \v1)]{(0.5, 0)};
    \vertexs[gray]{2}[(\h1,\v1)]{(-1, 0)};
    \vertexs[teal]{3}[(\h1,\v1)]{(-0.5,-0.866)};
    \vertexs[teal]{5}[(\h1,\v1)]{(-1,-1.732)};
    \vertexs[teal]{6}[(\h1,\v1)]{(0,-1.732)};
    \vertexs[teal]{7}[(\h1,\v1)]{(-0.5,-2.598)};
    \vertexs[teal]{8}[(\h1,\v1)]{(0,-3.464)};
    \vertexs[teal]{9}[(\h1,\v1)]{(1,-3.464)};
    \vertexs[teal]{10}[(\h1,\v1)]{(1.5,-2.598)};
    \vertexs[teal]{11}[(\h1,\v1)]{(1,-1.732)};
    \vertexs[teal]{12}[(\h1,\v1)]{(2,-1.732)};
    \vertexs[gray]{13}[(\h1,\v1)]{(2.5,-0.866)};
    \vertexs[gray]{14}[(\h1,\v1)]{(2,0)};
    \draw[color=teal] (v1)--(v3) (v3)--(v5)--(v7)--(v6)--(v3) (v5)--(v6) (v7)--(v8)--(v9);
    \draw[color=teal] (v10)--(v11) (v10)--(v12)--(v11) (v1)--(v11) (v1)--(v12) (v9)--(v10) (v11)--(v6);
    \draw[color=gray] (v13)--(v14);
    \draw[dashed, opacity=0.4] (v1)--(v2)--(v3) (v2)--(v5) (v12)--(v13) (v1)--(v13) (v1)--(v14);
    \node[align=center] at (\h1+0.75,\v1-4.6) {$\strip\in \ST(G^\circ)$,\\$\textcolor{teal}{\ctrip^\circ=(C^\circ=\varphi(C),\alpha=721, r=4)}$\\$k=r=4$};

    \tikzmath{\arrplace=-4.5; \varr=\v1-1.732;}
    \draw [arrows = {-Stealth[scale=2]}] (\arrplace-1.75,\varr)--(\arrplace+1.75,\varr);
    \draw [arrows = {-Stealth[scale=2]}] (\arrplace+1.75,\varr)--(\arrplace-1.75,\varr);

    \tikzmath{\h1=-10;}
    \vertex[red]{1}[(\h1, \v1)]{(0, 0)};
    \vertexs[gray]{2}[(\h1,\v1)]{(-1, 0)};
    \vertexs[teal]{3}[(\h1,\v1)]{(-0.5,-0.866)};
    \vertexs[teal]{5}[(\h1,\v1)]{(-1,-1.732)};
    \vertexs[teal]{6}[(\h1,\v1)]{(0,-1.732)};
    \vertexs[teal]{7}[(\h1,\v1)]{(-0.5,-2.598)};
    \vertexs[teal]{8}[(\h1,\v1)]{(0,-3.464)};
    \vertexs[teal]{9}[(\h1,\v1)]{(1,-3.464)};
    \vertexs[teal]{10}[(\h1,\v1)]{(1.5,-2.598)};
    \vertexs[teal]{11}[(\h1,\v1)]{(1,-1.732)};
    \vertexs[teal]{12}[(\h1,\v1)]{(2,-1.732)};
    \vertexs[gray]{13}[(\h1,\v1)]{(2.5,-0.866)};
    \vertexs[gray]{14}[(\h1,\v1)]{(2,0)};
    \vertex[red][-3pt][left]{15}[(\h1,\v1)]{(1,0)};
    \vertex[teal][-3pt][left]{16}[(\h1,\v1)]{(1,0.5)};
    \vertex[teal][-3pt][left]{17}[(\h1,\v1)]{(1,1)};
    \vertex[teal][-3pt][left]{18}[(\h1,\v1)]{(1,1.5)};
    \draw[color=teal] (v1)--(v3) (v3)--(v5)--(v7)--(v6)--(v3) (v5)--(v6) (v7)--(v8)--(v9);
    \draw[color=teal] (v10)--(v11) (v10)--(v12)--(v11) (v15)--(v11) (v15)--(v12) (v15)--(v16)--(v17)--(v18) (v9)--(v10) (v11)--(v6);
    \draw[color=gray] (v13)--(v14);
    \draw[dashed, opacity=0.4] (v1)--(v2)--(v3) (v2)--(v5) (v12)--(v13) (v15)--(v13) (v15)--(v14);
    \node[align=center] at (\h1+0.75,\v1-4.6) {$\strip\in \ST^{(4)}(G+P_4)$,\\$\textcolor{teal}{\ctrip=(C,\alpha=4217, r=1)}$\\$k=4$\\};
    \end{tikzpicture}
    \caption{\label{fig:tracepsi} Some examples of the map $\psi:\bigsqcup_k \ST^{(k)}(G+P_k) \to \ST(G^\circ)$.}
    \end{figure}
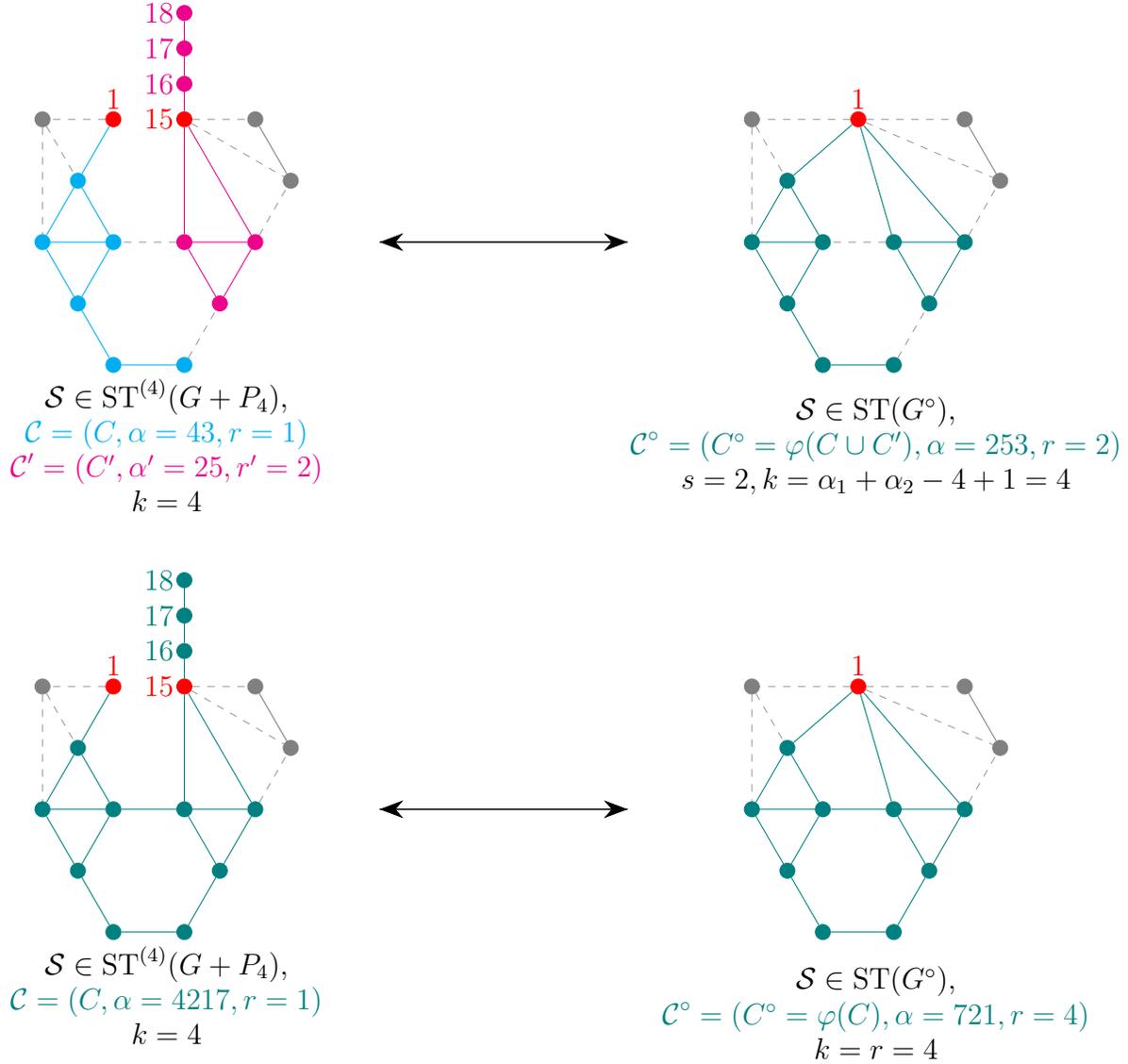

We define a bijection
\[\psi : \bigsqcup_k \ST^{(k)}(G+P_k) \to \ST(G^\circ)\]
as follows: given $\strip\in \ST^{(k)}(G+P_k)$, let $\ctrip = (C,\alpha, r)\in \strip$ and $\ctrip'=(C',\alpha',r') \in \strip$ where $1\in C, n\in C'$. If $\ctrip\neq \ctrip'$, then $\psi(\strip)$ replaces both component triples with
\[\ctrip^\circ = (\varphi(C\cup C'), \alpha'\cdot \alpha_2\cdots\alpha_{\ell}, r'),\]
which produces a subgraph triple in $\ST(G^\circ)$, as shown in the first row of Figure~\ref{fig:tracepsi}. If instead $\ctrip=\ctrip'$, then $\psi(\strip)$ replaces $\ctrip$ with
\[\ctrip^\circ=(\varphi(C), \alpha_{\ell}\cdot\alpha_2\cdots \alpha_{\ell-1}, \alpha_1).\]
Note in this case, $\alpha_1=k\leq \alpha_{\ell}$, so $\psi(\strip)$ is indeed a subgraph triple in $\ST(G^\circ)$. An example is shown in the second row in Figure~\ref{fig:tracepsi}.

The inverse is defined as follows: given $\strip\in \ST(G^\circ)$, let $\ctrip^\circ=(C^\circ,\alpha, r)\in \strip$ be the component with $1\in C^\circ$. Then, $\varphi^{-1}(C^\circ)$ has either 1 or 2 connected components. Say it has two components $C, C'$ (with vertex $1\in C, n\in C'$). Since $|\alpha|=|C^\circ|\geq |C'|$, there exists a minimal $s$ with $\alpha_1+\cdots +\alpha_s\geq |C'|$. Letting
\[k=\alpha_1+\cdots+\alpha_s-|C'|+1,\]
define $\varphi^{-1}(\strip)$ to replace $\ctrip^\circ$ with
\[\ctrip' = (C'\cup P_k, \alpha_1\cdots \alpha_s, r) \text{ and } \ctrip = (C, k\ \alpha_{s+1}\cdots \alpha_\ell, 1).\]
Note $\psi^{-1}(\strip)$ is a valid subgraph triple in $\ST^{(k)}(G+P_k)$ as the minimality of $s$ ensures $\alpha_1+\cdots+\alpha_{s-1}<|C'|$ and thus $\alpha_s\geq k$. If $\varphi^{-1}(C^\circ)$ instead has only 1 connected component $C$, then letting $k=r$ we define $\psi^{-1}(\strip)$ to replace $\ctrip^\circ$ with
\[\ctrip=(C, k\ \alpha_2\cdots \alpha_{\ell}\cdot \alpha_1, 1).\]
This is also a valid subgraph triple as $\alpha_1=k\leq \alpha_{\ell}$. The two cases correspond to the first and second row in Figure~\ref{fig:tracepsi} respectively.

Having found bijection $\psi$, letting $\strip^{\circ}=\psi(\strip)$ for $\strip\in\ST^{(k)}(G+P_k)$, note
\begin{gather*}
\sign(\strip^\circ)=\sign(\strip)\cdot (-1)^{1-k}, \; \; e_{\type(\strip^\circ)}=e_{\type'(\strip)}, \; \; E(\strip^\circ)=E(\strip)-k+1.
\end{gather*}
Then we have, as desired, that
\begin{align*}
    XB_{G^\circ}(\bm x; t) =& \sum_{\strip^\circ\in \ST(G^\circ)}\sign(\strip^\circ)\cdot t^{E(\strip^\circ)}\cdot e_{\type(\strip^\circ)}\\
    =&\sum_{k\geq 1}\left(\sum_{\strip\in\ST^{(k)}(G+P_k)}\sign(\strip)\cdot (-1)^{1-k}\cdot t^{E(\strip)-k+1}\cdot e_{\type'(\strip)}\right)\\
    =&\sum_{k \geq 1}(M_G)_{k,k}=\trace(M_G).
\end{align*}
\end{proof}

We can use this to find the Tutte symmetric function for $C_n$.
\begin{proposition}\label{prop:cycle-from-path}
We have
\[XB_{C_n}(\bm x; t)=\sum_{\alpha\models n}\alpha_1(\alpha_1+t)\cdots(\alpha_{\ell}+t)\cdot (-t)^{n-\ell(\alpha)}e_{\sort(\alpha)}.\]
\end{proposition}
\begin{proof}
Since $(P_{n+1})^\circ = C_n$, we have
    \begin{align*}XB_{C_n}(\bm x; t) &= \trace(M_{P_{n+1}}) = \sum_{k=1}^{n}\sum_{\substack{\alpha\models n\\\alpha_{\ell}\geq k}} (\alpha_1 + t)\cdots (\alpha_{\ell} + t)\cdot (-t)^{n-\ell(\alpha)}e_{\sort(\alpha)}\\
    &= \sum_{\alpha\models n}\alpha_{\ell}(\alpha_1+t)\cdots(\alpha_{\ell}+t)\cdot (-t)^{n-\ell(\alpha)}e_{\sort(\alpha)},\end{align*}
where we simplify the double sum by noting each composition $\alpha$ is counted precisely $\alpha_{\ell}$ times. Swapping $\alpha_{\ell}$ with $\alpha_1$ gives the final expression.
\end{proof}

\begin{remark}
    Similar to the formulas for $XB_{P_n}(\bm x; t)$ and $X_{P_n}(\bm x)$, the Tutte symmetric formula takes the chromatic symmetric formula for $C_n$, replaces the $-1$ with $t$, and adds some powers of $(-t)$.
\end{remark}

\begin{remark}
    Notice that given graphs $A, B, C$, then $(A+B+C)^\circ \cong (B+C+A)^\circ$ but is not necessarily isomorphic to $(A+C+B)^\circ$ (shown in Figure~\ref{fig:ABCACB}), analogous to how for matrices $A, B, C$, then $\trace(ABC)=\trace(BCA)$ but is not necessarily equal to $\trace(ACB)$.
\end{remark}
   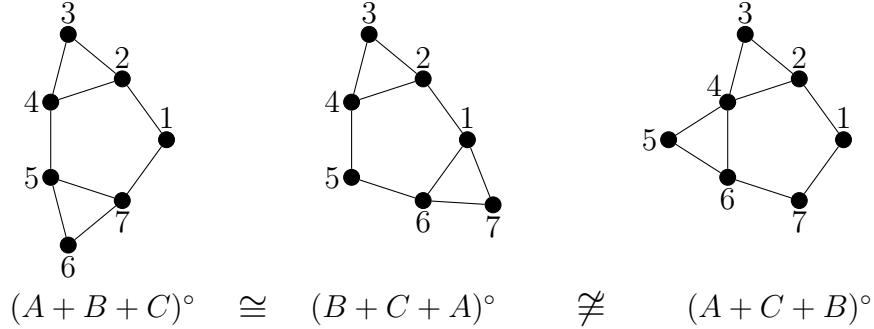
\begin{figure}
        \begin{tikzpicture}
\tikzmath{\v1=6; \h1=0; \r=0.851; \d=4;};
\vertex[black][-3pt][above]{1}[(\h1,\v1)]{(0:\r)};
\vertex[black][-3pt][above]{2}[(\h1,\v1)]{(72:\r)};
\vertex[black][-3pt][above]{3}[(\h1,\v1)]{(108:{sqrt(3)*\r})};
\vertex[black][-3pt][left]{4}[(\h1,\v1)]{(144:\r)};
\vertex[black][-3pt][left]{5}[(\h1,\v1)]{(-144:\r)};
\vertex[black][-3pt][below]{6}[(\h1,\v1)]{(-108:{sqrt(3)*\r})};
\vertex[black][-3pt][below]{7}[(\h1,\v1)]{(-72:\r)};
\draw (v1)--(v2)--(v3)--(v4)--(v2) (v4)--(v5)--(v6)--(v7)--(v5) (v7)--(v1);
\node[align=center] at (\h1,-2.25+\v1) {$(A+B+C)^\circ$};

\node[align=center, font=\large] at ({\h1+\d/2}, {\v1-2.25}) {$\cong$};

\tikzmath{\h1=\d;};
\vertex[black][-3pt][above]{1}[(\h1,\v1)]{(0:\r)};
\vertex[black][-3pt][above]{2}[(\h1,\v1)]{(72:\r)};
\vertex[black][-3pt][above]{3}[(\h1,\v1)]{(108:{sqrt(3)*\r})};
\vertex[black][-3pt][left]{4}[(\h1,\v1)]{(144:\r)};
\vertex[black][-3pt][left]{5}[(\h1,\v1)]{(-144:\r)};
\vertex[black][-3pt][below]{7}[(\h1,\v1)]{(-36:{sqrt(3)*\r})};
\vertex[black][-3pt][below]{6}[(\h1,\v1)]{(-72:\r)};
\draw (v1)--(v2)--(v3)--(v4)--(v2) (v4)--(v5)--(v6)--(v7)--(v1)--(v6);
\node[align=center] at (\h1,\v1+-2.25) {$(B+C+A)^\circ$};
\node[align=center, font=\large] at ({\h1+\d/2+1/2}, {\v1-2.25}) {$\not\cong$};
\tikzmath{\h1=\d+\h1+1;};
\vertex[black][-3pt][above]{1}[(\h1,\v1)]{(0:\r)};
\vertex[black][-3pt][above]{2}[(\h1,\v1)]{(72:\r)};
\vertex[black][-3pt][above]{3}[(\h1,\v1)]{(108:{sqrt(3)*\r})};
\vertex[black][-6pt][135]{4}[(\h1,\v1)]{(144:\r)};
\vertex[black][-3pt][below]{6}[(\h1,\v1)]{(-144:\r)};
\vertex[black][-3pt][left]{5}[(\h1,\v1)]{(180:{sqrt(3)*\r})};
\vertex[black][-3pt][below]{7}[(\h1,\v1)]{(-72:\r)};
\draw (v1)--(v2)--(v3)--(v4)--(v2) (v4)--(v5)--(v6)--(v4) (v6)--(v7)--(v1);
\node[align=center] at (\h1,\v1+-2.25) {$(A+C+B)^\circ$};
        \end{tikzpicture}
    \caption{\label{fig:ABCACB} Variants of $(A+B+C)^\circ$ where $A=P_2+K_3$, $B=P_2$, and $C=K_3+P_2$. Notice that vertex 4 in $(A+C+B)^\circ$ has degree 4 meaning $(A+C+B)^\circ\not\cong (A+B+C)^\circ$.}
    \end{figure}

\section{Reversing a Graph}\label{section:reverse}
In this section, we relate the operation of reversing a graph $G$ to $G_{\rev}$ (the graph with the vertex labelings reversed) with transposing $M_G$, up to a change of basis. Since $M_G$ depends on the labeling of the graph (in particular, the labels $1$ and $n$), certain graphs (like $G=K_2+K_3$) have $M_{G_{\rev}}\neq M_G$. More specifically, we have a function $f$ where $f(M_G)=M_{G_{\rev}}$. Since $G^\circ \cong G_{\rev}^\circ$ and $(G+H)_{\rev}\cong H_{\rev}+G_{\rev}$, function $f$ preserves the trace and is an antihomomorphism, meaning
\[\trace(f(M))=\trace(M), \hspace{1cm} f(PQ)=f(Q)f(P).\]

One natural family of functions with this property are transpositions up to change of basis (where $f(M)=K^{-1}M^TK$ for some fixed $K$). We will show the infinite symmetric matrix $K$ where
\[K_{i,j}=\frac{X_{K_i+K_j}(\bm x)}{(i-1)!(j-1)!},\hspace{20pt}
K = \left[
\begin{matrix}
e_1           & 2e_2           & 3e_3           & \cdots \\
2e_2          & e_{21}+3e_3    & 2e_{31}+4e_4   & \cdots \\
3e_3          & 2e_{31}+4e_4   & e_{32}+3e_{41}+5e_5 & \cdots \\
\vdots & \vdots & \vdots & \ddots
\end{matrix}
\right],\]
satisfies $Kf(M)= M^TK$. Interestingly, $K$ does not depend on $t$.

\begin{definition}
    For a graph $G = ([n], E)$, let
    \[G_{\rev}=([n], \{\{i, j\} : \{n + 1 - j, n + 1 - i\} \in E\}).\]

Note that $G\cong G_{\rev}$, meaning $XB_G(\bm x; t)=XB_{G_{\rev}}(\bm x; t)$. Additionally, 
\[(G_1+\cdots+G_k)_{\rev}=(G_k)_{\rev}+\cdots + (G_1)_{\rev}.\]
\end{definition}
\begin{example}
    Figure~\ref{fig:Grev_attach} shows an example of reversing a graph. Note reversal acts similar to transpose, where $(A+B)_{\rev}=B_{\rev}+A_{\rev}$.
\end{example}

  \begin{figure}
        \begin{tikzpicture}
\tikzmath{\d=4; \h1=\d; \v1=6; \r={sqrt(2)/2};};
\vertex[black][-3pt][above]{3}[(\h1,\v1)]{(45:\r)};
\vertex[black][-3pt][above]{2}[(\h1,\v1)]{(135:\r)};
\vertex[black][-3pt][below]{1}[(\h1,\v1)]{(-135:\r)};
\vertex[black][-3pt][below]{4}[(\h1,\v1)]{(-45:\r)};
\vertex[black][-3pt][above]{5}[(\h1,\v1)]{({0.5+sqrt(3)/2},0)};
\tikzmath{\h1={\d+0.5+sqrt(3)/2+1};\r=1;};
\vertex[black][-3pt][above]{6}[(\h1,\v1)]{(120:\r)};
\vertex[black][-3pt][below]{7}[(\h1,\v1)]{(-120:\r)};
\vertex[black][-3pt][above]{8}[(\h1,\v1)]{(60:\r)};
\vertex[black][-3pt][below]{9}[(\h1,\v1)]{(-60:\r)};
\vertex[black][-3pt][right]{10}[(\h1,\v1)]{(0:\r)};
\draw (v1)--(v2)--(v3)--(v4)--(v1) (v3)--(v5)--(v4) (v5)--(v6)--(v8)--(v10)--(v9)--(v7)--(v5) (v5)--(v8)--(v7)--(v6)--(v9)--(v5) (v8)--(v9) (v7)--(v10);
\node[align=center] at ({\h1-1},\v1+-2.25) {$C=A+B$};
\tikzmath{\h1=\h1+\d;\r=1;};
\vertex[black][-3pt][above]{5}[(\h1,\v1)]{(60:\r)};
\vertex[black][-3pt][below]{4}[(\h1,\v1)]{(-60:\r)};
\vertex[black][-3pt][above]{3}[(\h1,\v1)]{(120:\r)};
\vertex[black][-3pt][below]{2}[(\h1,\v1)]{(-120:\r)};
\vertex[black][-3pt][left]{1}[(\h1,\v1)]{(180:\r)};
\vertex[black][-3pt][above]{6}[(\h1,\v1)]{(1,0)};
\tikzmath{\h1={\h1+1+0.5+sqrt(3)/2};\r={sqrt(2)/2};};
\vertex[black][-3pt][above]{8}[(\h1,\v1)]{(135:\r)};
\vertex[black][-3pt][above]{9}[(\h1,\v1)]{(45:\r)};
\vertex[black][-3pt][below]{10}[(\h1,\v1)]{(-45:\r)};
\vertex[black][-3pt][below]{7}[(\h1,\v1)]{(-135:\r)};
\draw (v10)--(v9)--(v8)--(v7)--(v10) (v8)--(v6)--(v7) (v6)--(v5)--(v3)--(v1)--(v2)--(v4)--(v6) (v6)--(v3)--(v4)--(v5)--(v2)--(v6) (v3)--(v2) (v4)--(v1);
\node[align=center] at (\h1-1,\v1+-2.25) {$C_{\rev}=B_{\rev}+A_{\rev}$};
        \end{tikzpicture}
    \caption{\label{fig:Grev_attach} Examples of reversing a graph.}
\end{figure}
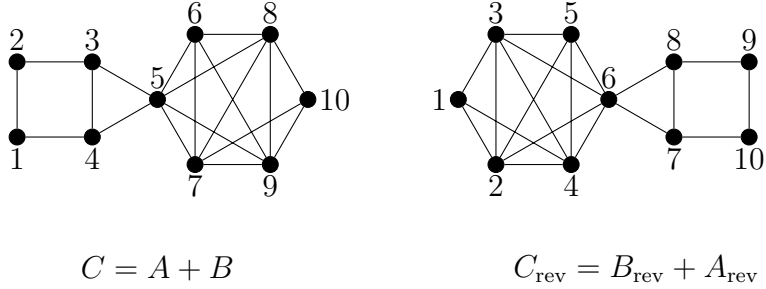

\begin{definition}
    A graph $G$ is \emph{reversible} if $G = G_{\rev}$, meaning $\{i,j\}\in E(G)$ if and only if $\{n-i+1, n-j+1\}\in E(G)$. This implies $H + G + J \cong H + (G_{\rev}) + J$ for all graphs $H, J$.
\end{definition}
\begin{example}
    Path graphs, cycle graphs, and cliques are all reversible graphs.
\end{example}

To relate $M_G$ with $M_{G_{\rev}}$, we will use the fact that for any two graphs $H_i, H_j$, \[XB_{H_i+G+(H_j)_{\rev}}(\bm x; t)=XB_{H_j+G_{\rev}+(H_i)_{\rev}}(\bm x; t).\]
Thus, given any family of graphs $\mathcal{H}=H_1,H_2,\ldots$, we can construct multiple equalities which will help relate $M_G$ to $M_{G_{\rev}}$. It turns out that so long as the family of graphs $\mathcal{H}$ satisfies certain rules, we always derive the same relation between $M_G$ and $M_{G_{\rev}}$.

\begin{definition}\label{def:rev:AB}
    A \emph{spanning family of graphs} $\mathcal{H}=H_1, H_2, \ldots$ is a sequence of connected, loopless graphs where $|H_i|=i$. Given a spanning family of graphs, we define matrices $A_{\mathcal{H}}$ and $B_{\mathcal{H}}$ where
    \[(A_{\mathcal{H}})_{i,j}=(\vec{v}M_{H_i})_j, \hspace{1cm} (B_{\mathcal{H}})_{i,j}=(M_{H_j}\vec{w}^T)_i=(M_{H_j})_{i, 1}.\]
    In particular, the $i$th row of $A_{\mathcal{H}}$ is the row vector $\vec{v}M_{H_i}$, and the $j$th column of $B_{\mathcal{H}}$ is the column vector $M_{H_j}\vec{w}^T$.
\end{definition}
\begin{remark}
    One natural reason to work with spanning families is that given any spanning family of connected simple graphs $\mathcal{H}$, the set of chromatic symmetric functions $\{X_{H_i}(\bm x)\}_i$ generates the algebra of symmetric functions \cite[Theorem~5]{chrombase}.
\end{remark}

\begin{definition}\label{def:rev:spanning_family}
    The reverse of a spanning family ${\mathcal{H}}$ is ${\mathcal{H}}_{\rev}=(H_1)_{\rev}, (H_2)_{\rev}, \ldots$. A \emph{reversible} spanning family ${\mathcal{H}}$ is one where every $H_i$ is \emph{reversible}, which implies $A_{\mathcal{H}}=A_{{\mathcal{H}}_{\rev}}$ and $B_{\mathcal{H}}=B_{{\mathcal{H}}_{\rev}}$.
\end{definition}
\begin{example}
    The sequence of cliques and paths ($H_i = K_i$ and $H_i = P_i$ respectively) form two reversible spanning families of graphs.
\end{example}

Our goal is to show the matrix $K=A_{\mathcal{H}}^TB_{{\mathcal{H}}_{\rev}}^{-1}$ is well-defined and the same for any spanning family $\mathcal{H}$.
\begin{proposition}\label{prop:rev:A_HB_Hrevsymm}
    For any spanning family $\mathcal{H}$, the matrix $A_{\mathcal{H}}B_{{\mathcal{H}}_{\rev}}$ is symmetric. If ${\mathcal{H}}$ is reversible, then $A_{\mathcal{H}}B_{\mathcal{H}}$ is symmetric.
\end{proposition}
\begin{proof}
    Note $XB_{H_i+(H_j)_{\rev}}(\bm x; t) = XB_{(H_i+(H_j)_{\rev})_{\rev}}(\bm x; t) = XB_{H_j + (H_i)_{\rev}}(\bm x; t)$, and
    \[(A_{\mathcal{H}}B_{{\mathcal{H}}_{\rev}})_{i,j}=(\vec{v}M_{H_i})(M_{{H_j}_{\rev}}\vec{w}^T)=XB_{H_i+(H_j)_{\rev}}(\bm x; t)=XB_{H_j + (H_i)_{\rev}}(\bm x; t) = (A_{\mathcal{H}}B_{{\mathcal{H}}_{\rev}})_{j,i}.\]
    If ${\mathcal{H}}$ is reversible then $\mathcal{H}_{\rev}=\mathcal{H}$ so we get $A_{\mathcal{H}}B_{\mathcal{H}}$ is symmetric.
\end{proof}
\begin{comment}
\begin{proposition}
    For any spanning family $H$, the upper left $m \times m$ matrix $(A_H)_m$ is invertible
\end{proposition}
\begin{proof}
    \textcolor{red}{todo}
\end{proof}
\end{comment}
\begin{proposition}\label{prop:rev:Binvertible}
    For any spanning family ${\mathcal{H}}$, the matrix $B_{\mathcal{H}}$ is upper triangular and invertible. In particular, there exists an upper triangular matrix $B_{\mathcal{H}}^{-1}$ where $B_{\mathcal{H}}\cdot B_{\mathcal{H}}^{-1} = B_{\mathcal{H}}^{-1}\cdot B_{\mathcal{H}} = I$.
\end{proposition}
\begin{proof}
    Given $G$ has $n$ vertices then $(M_G)_{i,1}=0$ if $i\geq n+1$ by Proposition~\ref{prop:ftmatrixzeroes}. Since $H_j$ has $j$ vertices, then $(B_{\mathcal{H}})_{i,j}=(M_{H_j})_{i,1}=0$ if $i\geq j+1$, so $B_{\mathcal{H}}$ is upper triangular. To show $B_{\mathcal{H}}$ is invertible it suffices to show $(B_{\mathcal{H}})_{i,i}\neq 0$ for all $i$. Note
    \[(B_{\mathcal{H}})_{i,i}=(M_{H_i})_{i,1}=\sum_{\substack{S \subseteq H_i\\\text{$S$ has 1 connected component}}}(-1)^{i-1}\cdot t^{E(S)}.\]
    Since $H_i$ is connected, there exists at least one such subgraph, meaning the sum is non-empty and $(B_{\mathcal{H}})_{i,i}\neq 0$.
\end{proof}

\begin{proposition}
    The matrix $K=A_{\mathcal{H}}^TB_{{\mathcal{H}}_{\rev}}^{-1}$ is symmetric and the same for any spanning family ${\mathcal{H}}$.
\end{proposition}
\begin{proof}
    By Proposition~\ref{prop:rev:A_HB_Hrevsymm} and Proposition~\ref{prop:rev:Binvertible}, we have $(B_{{\mathcal{H}}_{\rev}})^TA_{\mathcal{H}}^T=A_{\mathcal{H}}B_{{\mathcal{H}}_{\rev}}$, which implies $A_{\mathcal{H}}^T=(B_{{\mathcal{H}}_{\rev}}^T)^{-1}A_{\mathcal{H}}^TB_{{\mathcal{H}}_{\rev}}$. Then,
    \[K=A_{\mathcal{H}}^TB_{{\mathcal{H}}_{\rev}}^{-1}=(B_{{\mathcal{H}}_{\rev}}^T)^{-1}A_{\mathcal{H}}^TB_{{\mathcal{H}}_{\rev}}B_{{\mathcal{H}}_{\rev}}^{-1}=(B_{{\mathcal{H}}_{\rev}}^T)^{-1}A_{\mathcal{H}}^T=K^T\]
    is symmetric. Let ${\mathcal{H}}$ and ${\mathcal{J}}$ be two spanning families, note \[(A_{\mathcal{H}}B_{{\mathcal{J}}_{\rev}})_{i,j}=XB_{H_i+(J_j)_{\rev}}(\bm x; t)=XB_{J_j+(H_i)_{\rev}}(\bm x; t)=(A_{\mathcal{J}}B_{{\mathcal{H}}_{\rev}})_{j,i},\]
    meaning $A_{\mathcal{H}}B_{{\mathcal{J}}_{\rev}}=B_{{\mathcal{H}}_{\rev}}^TA_{\mathcal{J}}^T$ and $A_{\mathcal{H}}=B^T_{{\mathcal{H}}_{\rev}}A^T_{\mathcal{J}}B_{{\mathcal{J}}_{\rev}}^{-1}$. Then,
    \[K=A_{\mathcal{H}}^TB^{-1}_{{\mathcal{H}}_{\rev}}=(B_{{\mathcal{J}}_{\rev}}^T)^{-1}A_{\mathcal{J}}B_{{\mathcal{H}}_{\rev}}B_{{\mathcal{H}}_{\rev}}^{-1}=(B^T_{{\mathcal{J}}_{\rev}})^{-1}A_{\mathcal{J}}=A_{\mathcal{J}}^TB_{{\mathcal{J}}_{\rev}}^{-1},\]
    where the last step uses the fact that $A_{\mathcal{J}}^TB_{{\mathcal{J}}_{\rev}}^{-1}$ is symmetric. Thus, for any pair of spanning families ${\mathcal{H}}$ and ${\mathcal{J}}$, the resulting $K$ matrix is the same.
\end{proof}

\begin{theorem}
    For any graph $G$ and its reverse $G_{\rev}$, we have
    \[KM_{G_{\rev}}=M_G^TK.\]
\end{theorem}
\begin{proof}
    Let $H_i=P_i$ for simplicity, meaning ${\mathcal{H}}={\mathcal{H}}_{\rev}$, then
    \begin{align*}(A_{\mathcal{H}}M_{G_{\rev}}B_{\mathcal{H}})_{i,j}&=\vec{v}M_{P_i}M_{G_{\rev}}M_{P_j}\vec{w}^T=XB_{P_i+G_{\rev}+P_j}(\bm x; t)\\
    &=XB_{P_j+G+P_i}(\bm x; t)=(A_{\mathcal{H}}M_{G}B_{\mathcal{H}})_{j,i},\end{align*}
    meaning $A_{\mathcal{H}}M_{G_{\rev}}B_{\mathcal{H}} = B^T_{\mathcal{H}}M_{G}^TA_{\mathcal{H}}^T$. By Proposition~\ref{prop:rev:Binvertible}, we can invert $B_{\mathcal{H}}$ to get
    \[KM_{G_{\rev}}=(B^T_{\mathcal{H}})^{-1}A_{\mathcal{H}}M_{G_{\rev}}=M_{G}^TA_{\mathcal{H}}^TB_{\mathcal{H}}^{-1}=M_{G}^TK,\]
    which proves the theorem.
\end{proof}

Using $H_i=P_i$, we will now calculate $A_{\mathcal{H}}$ and $B_{\mathcal{H}}^{-1}$ to find $K$.

\begin{proposition}
    For $H_i=P_i$, then 
    \[(B_{\mathcal{H}}^{-1})_{i,j}=\begin{cases}
        (-t)^{-i}(i-j-t)e_{j-i}&i\leq j \\
        0 &i > j
    \end{cases}.\]
\end{proposition}
\begin{proof}
    Since $B_{\mathcal{H}}$ is upper triangular, its inverse is also upper triangular and showing $B_{\mathcal{H}}^{-1}B_{\mathcal{H}}=I$ implies $B_{\mathcal{H}}B_{\mathcal{H}}^{-1}=I$. Note by Proposition~\ref{prop:ftpn},
    \[(B_{\mathcal{H}})_{i,j}=(M_{H_j})_{i,1}=\sum_{\alpha \models j-i}(\alpha_1+t)\cdots(\alpha_\ell+t)(-t)^{j-\ell(\alpha)-1}e_{\sort(\alpha)},\]
    meaning
    \begin{align*}
        (B_{\mathcal{H}}B_{\mathcal{H}}^{-1})_{i,j}&=\sum_{k=i}^j\left(\sum_{\alpha\models k-i}(\alpha_1+t)\cdots(\alpha_\ell+t)(-t)^{k-\ell(\alpha)-1}e_{\sort(\alpha)}\right)\cdot (-t)^{-k}(k-j-t)e_{j-k}\\
        &=\sum_{k=i}^j\sum_{\alpha\models k-i}(\alpha_1+t)\cdots(\alpha_\ell+t)(k-j-t)(-t)^{-\ell(\alpha)-1}e_{\sort(\alpha)}e_{j-k}.
    \end{align*}
    If $i > j$ then $(B_{\mathcal{H}}B_{\mathcal{H}}^{-1})_{i,j}=0$. If $i < j$, we can combine the sum over $k$ and the sum over $\alpha\models k-i$ into a single sum over $\alpha\models j-i$ with $\alpha_1=j-k$, getting
    \begin{align*}(B_{\mathcal{H}}B_{\mathcal{H}}^{-1})_{i,j}=&\sum_{\alpha\models j-i}(\alpha_1+t)\cdots(\alpha_\ell+t)(-t)^{-\ell(\alpha)}e_{\sort(\alpha)}\\+&\sum_{k=i}^{j-1}\sum_{\alpha\models k-i}(\alpha_1+t)\cdots(\alpha_\ell+t)(k-j-t)(-t)^{-\ell(\alpha)-1}e_{\sort(\alpha)}e_{j-k}\\
    =&\sum_{\alpha\models j-i}(\alpha_1+t)\cdots(\alpha_\ell+t)(-t)^{-\ell(\alpha)}e_{\sort(\alpha)}\\
    -&\sum_{\substack{\alpha\models j-i\\j-k=\alpha_{1}}}(\alpha_1+t)\cdots (\alpha_{\ell}+t)(-t)^{-\ell(\alpha)}e_{\sort(\alpha)}
    = 0.
    \end{align*}
    Finally, if $i=j$, then the sum accounts for only the empty composition, resulting in $(B_{\mathcal{H}}B_{\mathcal{H}}^{-1})_{i,j}=1$. Thus, $B_{\mathcal{H}}B_{\mathcal{H}}^{-1}$ is the identity matrix.
\end{proof}
\begin{proposition}
    For $H_i = P_i$, then
    \[(A_{\mathcal{H}})_{i,j}=\sum_{\substack{\alpha\models i+j-1\\\alpha_\ell\geq j}}\alpha_1\cdot (\alpha_2+t)\cdots(\alpha_\ell+t)(-t)^{i-\ell(\alpha)}e_{\sort(\alpha)}.\]
\end{proposition}
\begin{proof}
Using Definition~\ref{def:rev:AB}, $(A_{\mathcal{H}})_{i,j}=(\vec{v}M_{P_i})_j$ meaning
\begin{align*}
    (A_{\mathcal{H}})_{i,j}=&\sum_k \vec{v}_k \cdot (M_{P_i})_{k,j}\\
    &=\sum_k \left(k e_k \cdot \sum_{\substack{\alpha\models i+j-k-1\\\alpha_\ell\geq j}}(\alpha_1+t)\cdots(\alpha_\ell+t)(-t)^{i-\ell(\alpha)-1}e_{\sort(\alpha)}\right)\\
    &=\sum_{\substack{\alpha \models i+j-1\\\alpha_{\ell}\geq j}}\alpha_1 (\alpha_2+t)\cdots(\alpha_{\ell}+t)(-t)^{i-\ell(\alpha)}e_{\sort(\alpha)}.
\end{align*}
\end{proof}
\begin{theorem}
    The matrix $K$ has $K_{i,j}=\frac{X_{K_i+K_j}(\bm x)}{(i-1)!(j-1)!}$.
\end{theorem}
\begin{proof}
    Note $K=A_{\mathcal{H}}^T(B_{{\mathcal{H}}}^{-1})$ when $H_i=P_i$, so $K_{i,j}$ is
    \begin{align*}
        K_{i,j}=&\sum_k (A_{\mathcal{H}})_{k,i}(B_{\mathcal{H}}^{-1})_{k,j}\\
        =&\sum_{k=1}^j \sum_{\substack{\alpha \models k+i-1\\\alpha_{\ell}\geq i}}\alpha_1(\alpha_2+t)\cdots (\alpha_{\ell}+t)(-t)^{k-\ell(\alpha)}e_{\sort(\alpha)}\cdot (-t)^{-k}(k-j-t)e_{j-k}\\
        =&\sum_{k=1}^{j-1}\sum_{\substack{\alpha \models k+i-1\\\alpha_{\ell}\geq i}}-\alpha_1(\alpha_2+t)\cdots(\alpha_{\ell}+t)(-t)^{-\ell(\alpha)}(j-k+t)e_{\sort(\alpha)}e_{j-k} \\
        +&\sum_{\substack{\alpha\models j+i-1\\\alpha_{\ell}\geq i}}\alpha_1(\alpha_2+t)\cdots(\alpha_{\ell}+t)(-t)^{1-\ell(\alpha)}e_{\sort(\alpha)}.
    \end{align*}
    The next step is to transform the double sum over $k$ and $\alpha\models k+i-1$ into a single sum over $\alpha'\models j+i-1$. To do this, we look at $\alpha\models k+i-1$ where $\len(\alpha)=1$ and $\len(\alpha)\geq 2$ separately, meaning the new composition $\alpha' \models j+i-1$ has either $\len(\alpha')=2$ (where $\alpha'_1=j-k$) or $\len(\alpha')\geq 3$ (where $\alpha'_2=j-k$), getting
    \begin{align*}
        K_{i,j}=&\sum_{\substack{\alpha \models j+i-1\\\len(\alpha)=2\\\alpha_{\ell}\geq i}}-(\alpha_1+t)\alpha_2(-t)^{1-\ell(\alpha)}e_{\sort(\alpha)} \\
        -&\sum_{\substack{\alpha \models j+i-1\\\len(\alpha)\geq 3\\\alpha_{\ell}\geq i}}\alpha_1(\alpha_2+t)\cdots (\alpha_\ell+t)(-t)^{1-\ell(\alpha)}e_{\sort(\alpha)} \\
        +&\sum_{\substack{\alpha \models j+i-1\\\alpha_{\ell}\geq i}}\alpha_1(\alpha_2+t)\cdots(\alpha_{\ell}+t)(-t)^{1-\ell(\alpha)}e_{\sort(\alpha)}\\
        =&\sum_{\substack{\alpha \models j+i-1\\\len(\alpha)=2\\\alpha_{\ell}\geq i}}\left(\alpha_1(\alpha_2+t)-(\alpha_1+t)\alpha_2\right)(-t)^{-1}e_{\sort(\alpha)} + \sum_{\substack{\alpha \models j+i-1\\\len(\alpha)=1}}\alpha_1 e_{\sort(\alpha)}\\
        =&\hspace{0.3cm}\sum_{k=i}^{i+j-1}(2k-i-j+1)e_{k}e_{i+j-k-1}\\
        =&\sum_{k=\max(i,j)}^{i+j-1}(2k-i-j+1)e_{k}e_{i+j-k-1} = \frac{X_{K_i+K_j}(\bm x)}{(i-1)!(j-1)!}.
    \end{align*}
    In the second to last line, we first replace the sum over compositions with a sum over $k$ where $k$ represents $\alpha_{\ell}$, and then combine the two sums. In the last line, we note that if $i < j$ then the terms for $k < j$ cancel out with each other. This final expression for $K_{i,j}$ equals the chromatic symmetric function of $K_i+K_j$ up to a scaling factor, calculated in \cite[Corollary~4.14]{qforesttriples}.

    Notice that all factors of $t$ cancel out and that $K_i+K_j$ is $e$-positive, meaning $K$ is a symmetric matrix with non-negative coefficients in every entry.
\end{proof}
\begin{proposition}
    If $G$ is reversible, then $KM_G=M_{G_{\rev}}K$ is symmetric.
\end{proposition}
\begin{proof}
    If $G$ is reversible then $M_G=M_{G_{\rev}}$, and since $K=K^T$, we have 
    \[KM_G = KM_{G_{\rev}}=(M_G)^TK=(KM_G)^T.\]
\end{proof}
\begin{example}
    Since $P_2$ is reversible, setting $t=-1$ for brevity we have
    \[KM_{P_2}=\begin{bmatrix}
e_1 & 2e_2 & 3e_3 & 4e_4 & \cdots \\
2e_2 & e_{21}+3e_3 & 2e_{31}+4e_4 & 3e_{41}+5e_5 & \cdots \\
3e_3 & 2e_{31}+4e_4 & e_{32}+3e_{41}+5e_5 & 2e_{42}+4e_{51}+6e_6 & \cdots \\
4e_4 & 3e_{41}+5e_5 & 2e_{42}+4e_{51}+6e_6 & e_{43}+3e_{52}+5e_{61}+7e_7 & \cdots \\
\vdots & \vdots & \vdots & \vdots & \ddots
\end{bmatrix} = (KM_{P_2})^T.\]
\end{example}

\printbibliography

\end{document}